\documentclass[reqno]{amsart}%默认值为reqno，表示公式编号放在公式右边，对齐在页面右边界处；如果使用选项leqno，表示公式编号放在公式左侧。
\usepackage{stmaryrd}
\usepackage{mathrsfs}
\usepackage{cases}
\usepackage{amsfonts}
\usepackage{amssymb}
\usepackage{amsmath}
\usepackage{tikz}
\usepackage{enumerate}
\newtheorem{theorem}{Theorem}[section]
\newtheorem{lemma}[theorem]{Lemma}
\newtheorem{proposition}[theorem]{Proposition}
\newtheorem{corollary}[theorem]{Corollary}
\theoremstyle{definition}

\newtheorem{remark}[theorem]{Remark}
\numberwithin{equation}{section}

\newcommand{\blankbox}[2]

\begin{document}
\title[The necessity of bounded commutators]
{The unified theory for the necessity of bounded commutators and applications}
\author{WEICHAO GUO}
\address{School of Mathematics and Information Sciences, Guangzhou University, Guangzhou, 510006, P.R.China}
\email{weichaoguomath@gmail.com}
\author{JIALI LIAN}
\address{School of Mathematical Sciences, Xiamen University,
Xiamen, 361005, P.R. China} \email{lianjiali@stu.xmu.edu.cn}
\author{HUOXIONG WU}
\address{School of Mathematical Sciences, Xiamen University,
Xiamen, 361005, P.R. China} \email{huoxwu@xmu.edu.cn}

\begin{abstract}
The general methods which are powerful for the necessity of
bounded commutators are given. As applications, some necessary conditions for bounded commutators are first obtained in certain endpoint cases, and several new characterizations of $BMO$ spaces, Lipschitz spaces and their weighted versions via boundedness of commutators in various function spaces are deduced.
\end{abstract}
\thanks{Supported by the NSF of China (Nos.11771358, 11471041, 11701112, 11671414), the NSF of Fujian Province of China (No.2015J01025)
and the China postdoctoral Science Foundation (No. 2017M612628).}

\maketitle

\section{Introduction}
Let $\mathscr {S}(\mathbb{R}^{n})$ be the Schwartz space and $\mathscr {S}^{\prime }(\mathbb{R}^{n})$ be the space
of tempered distributions.
Let $b$ be a local integrable function and $T$ be a linear operator
from $\mathscr{S}(\mathbb{R}^n)$ to $\mathscr{S}'(\mathbb{R}^n)$. The commutator $[b,T]$ generated by $T$ with
$b$ is defined as follows:
\begin{equation*}
  [b,T](f)=bT(f)-T(bf)
\end{equation*}
for suitable $f$, where $b$ is usually called the symbol of $[b,T]$.

Similarly, for a $m$-linear operator $\bar{T}$ from $\mathscr{S}(\mathbb{R}^n)\times\,\cdots\,\times\mathscr{S}(\mathbb{R}^n)$ to $\mathscr{S}'(\mathbb{R}^n)$, the $i$-th commutator associated with $\bar{T}$ and $b$ is defined by
$$[b,\bar{T}]_i(f_1,\cdots,f_m)=b\bar{T}(f_1,\cdots,f_i,\cdots,f_m)-\bar{T}(f_1,\cdots,bf_i, \cdots,f_m),\quad i=1,\cdots,m.$$

When studying the boundedness of commutator $[b, T]$ on function spaces (for example, Lebesgue spaces, Morrey spaces and their weighted versions, etc.), one usually has a so-called "upper bounded" result of the form: $\|[b,T]: X\to Y\|\lesssim \|b\|_*$ valid for all Calderon-Zygmund operators, certain rough singular integral operators etc., and a converse "lower bounded" result: $\|b\|_*\lesssim ||[b,T]:\,X\to Y\|$ usually valid only for a very nice subclass of Calderon-Zygmund operators (like the Riesz transforms, Riesz potential, et al.), where $\|\cdot\|_*$ denotes the $BMO$ norm or Lipschitz norm, $X,\,Y$ may be certain function spaces.
So does the multilinear version $[b, \bar{T}]_i$. Note that there still exist various gaps to be filled in the investigation of "lower bounded" result, which is also called "the necessity of bounded commutators".

The purpose of this paper is to try to fill the gaps left in the previous investigations. To do this, we will establish a general theory for the necessity of bounded commutators in linear and multi-linear settings. Our motivations also come from the following several aspects:
\begin{enumerate}
 % \item Until now, the boundedness of commutators $[b,T]$, or $[b, \bar{T}]_i$ can be established with very weak assumptions on the corresponding kernel.
 %In stark contrast, the characterization of
 % bounded commutator still need very strong assumption on the corresponding kernel.
 % Thus, a natural and interesting question is: can the bounded $[b,T]$ be characterized by $b$ with weaken assumption on the corresponding kernel?
  \item In order to give the necessity of bounded commutators, the methods used in most of the previous works are of two forms.
  The first one originated from Jason's work \cite{Jan} is by expanding the kernel locally by Fourier series, which seems to be convenient
  but leads to a very strong assumption on the corresponding kernel.
  The second one begins from dividing the commutator into main term and error term, which is origin from the technique used by Uchiyama in \cite{Uchiyama_Tohoku_1978}. However, this second method still need the (first order) smoothness of kernel,
  and it also need some tedious calculations in applications. Thus, a more efficient and useful method would be quite helpful and necessary.
  \item Note that most of the previous proofs for the necessity of bounded commutators are similar. Very recently, Chaffee and Cruz-Uribe \cite{Chaf-Cruz-2017} established the necessity of bounded commutators in a general Banach space structure. However, their method does not work in the Quasi-Banach space cases and endpoint cases. Moreover, their method need a very strong assumption on the corresponding kernel.
  Thus, it is very interesting to give a unified framework, which only suitable for the kernels with weak assumptions,
  but also still works on endpoint cases and Quasi-Banach spaces cases.
\end{enumerate}

This paper is organized as follows. In Section 2, after establishing a general criterion (see Theorem \ref{theorem, structure, homogeneous kernel}) for the necessity of bounded linear commutators in the general Quasi-Banach spaces, we will give some technique propositions for how to deal with the certain typical cases. Section 3 is concern with the multilinear commutators, in which the general criterion (see Theorem \ref{Theorem, structure, standard kernel}) and some technique propositions for the necessity of bounded $m$-linear commutators will be given. Finally, in Section 4 we will present a variety of applications, which essentially improve and extend previous results, and lead to certain new characterizations of $BMO$, Lipschitz spaces and their weighted versions via boundedness of commutators.

\section{The general structure and technique theorems in linear setting}
In this section, we give the necessity of bounded linear commutator in a very general structure.
Firstly, we give some basic assumptions of function spaces.
Throughout this paper, all the function spaces $X$ with quasi-norm $\|\cdot\|_X$ satisfy the following basic assumptions.
\\
\textbf{Basic assumptions}:

\begin{enumerate}[(i)]
  \item $\|f\|_X=\||f|\|_X$;
  \item if $|f|\leqslant |g|$ a.e., then $\|f\|_X\leqslant \|g\|_X$;
  \item if $\{f_n\}$ is a sequence of $X$ such that $|f_n|$ increases to $|f|$ a.e., then
  $\|f_n\|_X$ increases to $\|f\|_X$;
  \item if $A$ is a bounded set of $\mathbb{R}^n$, then $\|\chi_A\|_X<\infty$.
\end{enumerate}

Let $\mu$ be a positive function defined on all cubes of $\mathbb{R}^n$.
A locally integrable function $b$ is said to belong to function space $BMO_{\mu}$ if
\begin{equation*}
  \|b\|_{BMO_{\mu}}:=\sup_{Q}\frac{1}{\mu(Q)}\int_{B}|b(x)-b_B|dx<\infty,
\end{equation*}
where $b_Q=|Q|^{-1}\int_Qb(x)dx$, $|Q|$ denotes the Lebesgue measure of $Q$, $Q$ denotes the cube with sides parallel to the axes.

When $\beta\in (0,1]$, $\mu(Q)=\omega(Q)^{1+\beta/n}$ with some weight function $\omega$,
we also write $BMO_{\mu}={\rm Lip}_{\beta,\omega}$. In this case, ${\rm Lip}_{\beta,\omega}$ is called weighted Lipschitz space.
Especially, when $\mu(Q)=|Q|^{1+\beta/n}$,
Meyers \cite{Meyers_Proc.Amer.Math.Soc._1964} showed the following equivalent relation:
\begin{equation*}
  {\rm Lip}_{\beta}(\mathbb{R}^n)=BMO_{\mu}\ \text{with}\ \mu(Q)=|Q|^{1+\beta/n}.
\end{equation*}
In general, the following relations hold:
\begin{equation}
  BMO_{\mu}=
  \begin{cases}
    BMO, & if\ \mu(Q)=|Q|,\\
    {\rm Lip}_{\beta}, & if\ \mu(Q)=|Q|^{1+\beta/n}, \beta\in (0,1],\\
    {\rm Lip}_{\beta,\omega}, & if\ \mu(Q)=\omega(Q)^{1+\beta/n}, \beta\in (0,1].
  \end{cases}
\end{equation}

Very recently, Chaffee and Cruz-Uribe \cite{Chaf-Cruz-2017} established the following general result.

\medskip

{\bf Theorem A} (\cite{Chaf-Cruz-2017}).\quad{\it Given Banach function spaces $X$ and $Y$, $0\le \alpha<n$. Suppose that for every cube $Q$,
$$|Q|^{-\alpha/n}\|\chi_Q\|_{Y'}\|\chi_Q\|_X\lesssim |Q|.$$
Let $T$ be a linear operator defined on $X$, which can be represented by
$$Tf(x)=\int_{\mathbb{R}^n}K(x-y)f(y)dy$$
for all $x\notin {\rm supp}(f)$, where $K$ is a homogeneous kernel of degree $-n+\alpha$. Suppose further that there exists a ball $B\subset\mathbb{R}^n$ on which $1/K$ can be expanding by absolutely convergent Fourier series. If the commutator $[b,T]$ is bounded from $X$ to $Y$, then $b\in BMO(\mathbb{R}^n)$.}

\medskip

In what follows, we will relax the conditions and conclusion in Theorem A to deal with more general cases and fill various gaps in previous works. For brevity, the linear operators $T$ we are interested here are of the form:
\begin{equation}\label{integral operator, linear case}
  T_{\alpha}f(x)=\int_{\mathbb{R}^n}\frac{\Omega(x-y)}{|x-y|^{n-\alpha}}f(y)dy,
\end{equation}
where $\alpha\in [-1,n)$,
$\Omega$ is a homogeneous function of degree zero, having the following mean value zero property when $\alpha\leq 0$:
\begin{equation}\label{mean value zero}
  \int_{\mathbb{S}^{n-1}}\Omega(x')d\sigma(x')=0.
\end{equation}

When $\alpha=-1$, the commutator $[b,T_{-1}]$ was first investigated by Cald\'{e}ron \cite{Calderon-1965}                              in 1965 and now is well known as Cald\'{e}ron first-order commutator. When $\alpha\in (-1,0)$, very recently Chen, Ding and Hong \cite{CDH-2016} first established some new results of $[b,T_\alpha]$. When $\alpha=0$, $[b,T_0]$ is the commutator of the classical singular integral operators, which was first studied by Coifman-Rochberg-Weiss \cite{CRW-1976} in 1976. When $0<\alpha<n$, $[b,T_\alpha]$ is the commutator of fractional integral operators with homogenous kernels.

Now, we give a unified theory for the necessity of $[b, T_\alpha]$, which can be formulated as follows:

\begin{theorem}[Structure, homogeneous kernel]\label{theorem, structure, homogeneous kernel}
Let $X, Y$ be a pair of Quasi-Banach spaces, and $\mu$ be a positive function defined on all cubes (with sides parallel to the axes) in $\mathbb{R}^n$,
satisfying
\begin{equation*}
  \|\chi_{Q}\|_X\lesssim \|\chi_{Q}\|_Y\mu(Q)|Q|^{\alpha/n-1} \text{for all cubes}\ Q,
\end{equation*}
and one of the following conditions:
\begin{enumerate}[(a)]
  \item $\|\chi_{\lambda Q}\|_Y\leq C_{\lambda}\|\chi_Q\|_Y$ for all $\lambda>1$ and all cubes $Q$,\\
  \item $\|\chi_{\lambda Q}\|_X\leq C_{\lambda}\|\chi_Q\|_X$, $\mu(\lambda Q)\leq C_{\lambda}\mu(Q)$ for all $\lambda>1$ and all cubes $Q$.
\end{enumerate}
  Let $b\in L_{loc}^1(\mathbb{R}^n)$ and $\Omega$ be a function satisfying homogeneous condition of degree zero.
  Let $T_{\alpha}$ be the integral operator defined in (\ref{integral operator, linear case})
  and $[b,T_{\alpha}]$ be the corresponding commutator associated with $b$ and $T_{\alpha}$.
  Suppose $[b,T_{\alpha}]$ is a bounded operator from $X$ to $Y$.
  If we can find two Quasi-Banach spaces $\widetilde{Y}$ and $Z$ satisfying
  \begin{enumerate}[(I)]
  \item $ Y\cdot Z\subset\widetilde{Y}$,
  $\|\chi_{Q}\|_{\widetilde{Y}}\sim \|\chi_{Q}\|_{Y}\cdot \|\chi_{Q}\|_{Z}$ for every cube $Q$,
  \item $\frac{\|f(a\cdot+b)\|_Z}{\|\chi_{Q_0}(a\cdot+b)\|_Z}
  \sim \frac{\|f(\cdot)\|_Z}{\|\chi_{Q_0}(\cdot)\|_Z}$ for all $f\in Z$, where $Q_0=[-1/2,1/2]$,
\end{enumerate}
  such that $\Omega$ satisfies the following local property in $\mathbb{S}^{n-1}$: there exists an open subset of $\mathbb{S}^{n-1}$, denoted by $E$, such that
  \begin{enumerate}[(1)]
  \item lower and upper bound: \\$c\leqslant \Omega(x') \leqslant C$ for all $x'\in E$, where $0<c<C$, or $c<C<0$,
  \item for any open subset $F\subset E$, there exists a sequence $\{h_l\}_{l=1}^{\infty}$
  satisfying that $h_l':=h_l/|h_l|\in F$, $|h_l|\rightarrow \infty$ as $l\rightarrow \infty$, and
  \begin{equation*}
    \lim\limits_{\substack{l\rightarrow \infty}}
      \left\|\int_{Q_0}|\Omega(\cdot-y+h_l)-\Omega(\cdot+h_l)|dy\right\|_{Z(Q_0)}\rightarrow 0.
  \end{equation*}
  \end{enumerate}
  Then $b\in BMO_{\mu}$, and $\|b\|_{BMO_{\mu}}\lesssim \|[b,T_{\alpha}]\|_{X\rightarrow Y}$.
\end{theorem}

\begin{proof}
Without loss of generality, we may assume $0<c<C$.
Since there is no confusion, $T_{\alpha}$ will be abbreviated to $T$ in this proof.
By the local boundedness of $\Omega$ and the homogeneous condition of degree zero, we can find a constant $\tau_0$ and a nonempty open cone $\Gamma\subset \mathbb{R}^n$ with vertex at the origin,
such that for any $u\in 2Q_0$, $v\in \Gamma_{\tau_0}=\Gamma\cap B^c(0,\tau_0)$, we have $(u+v)':=(u+v)/|u+v|\in E$, and
  \begin{equation}
    \Omega(u+v)\sim 1, \ |u+v|\sim |v|.
  \end{equation}
Furthermore, we can find a sequence $\{h_l\}_{l=1}^{\infty}$ satisfying that $h_l\in \Gamma_{\tau_0}$, and $|h_l|\rightarrow \infty$
as $l\rightarrow \infty$, such that
\begin{equation}\label{for proof, structure, 10}
  \lim\limits_{{l\rightarrow \infty}}
      \left\|\int_{Q_0}|\Omega(\cdot-y+h_l)-\Omega(\cdot+h_l)|dy\right\|_{Z(Q_0)}\rightarrow 0.
\end{equation}
For a fixed cube $Q_1$ with side length $\rho$, we denote
\begin{equation}
  B=\frac{1}{\mu(Q_1)}\int_{Q_1}|b(y)|dy.
\end{equation}
Without loss of generality, we assume
\begin{equation*}
  \int_{Q_1}b(y)dy=0.
\end{equation*}
Take
\begin{equation*}
  \phi(x)=\Big(sgn(b)(x)-\frac{1}{|Q_1|}\int_{Q_1}sgn(b)(y)dy\Big)\chi_{Q_1}(x).
\end{equation*}
Then $-2\chi_{Q_1}\leq \phi\leq 2\chi_{Q_1}$, $b\phi\geqslant 0$.
  Denote $Q_{(l)}=Q_1+\rho h_l$ for $h_l\in \Gamma_{\tau_0}$ and
  let $Q_{M}=Q_{(l)}\cap \{x\in \mathbb{R}^n: b(x)\leqslant M\}$, where $M>0$.
  We have $\|b\chi_{Q_{M}}\|_Y<\infty$ for any $M>0$.
  By the assumption $Y\cdot Z\subset \widetilde{Y}$,
  we have
\begin{equation*}
  \begin{split}
    \|[b,T]\phi\|_Y
    \geqslant
    \|[b,T]\phi\chi_{Q_{M}}\|_Y
    \gtrsim
    \|[b,T]\phi\chi_{Q_{M}}\|_{\widetilde{Y}}/\|\chi_{Q_{M}}\|_Z.
  \end{split}
\end{equation*}
Note that there exist constants $A_1$ and $A_2$, which may be changed line to line, such that
\begin{equation*}
  \|[b,T]\phi\chi_{Q_{M}}\|_{\widetilde{Y}}
  \geq
   A_1\|T(b\phi)\chi_{Q_{M}}\|_{\widetilde{Y}}-A_2\|bT(\phi)\chi_{Q_{M}}\|_{\widetilde{Y}}.
\end{equation*}
We get
\begin{equation}\label{for proof, structure, 1}
  \|[b,T]\phi\|_Y
  \geqslant
  \frac{A_1\|T(b\phi)\chi_{Q_{M}}\|_{\widetilde{Y}}-A_2\|bT(\phi)\chi_{Q_{M}}\|_{\widetilde{Y}}}{\|\chi_{Q_M}\|_Z},
\end{equation}
For $x\in Q_{(l)}$ and $y\in Q_1$, we have $$(x-y)/\rho\in 2Q_0+h_l\subset 2Q_0+\Gamma_{\tau_0}.$$
Thus,
\begin{equation}\label{for proof, structure, 4}
  \Omega(x-y)=\Omega\left(\frac{x-y}{\rho}\right)\sim 1,\  \left|\frac{x-y}{\rho}\right|\sim |h_l|, \quad\forall\,x\in Q_{(l)},\, y\in Q_1.
\end{equation}
Recalling $b\phi\geqslant 0$, we obtain
\begin{equation*}
  \begin{split}
    |T(b\phi)(x)|=&\int_{Q_1} \frac{\Omega(x-y)}{|x-y|^{n-\alpha}}b(y)\phi(y)dy
    \gtrsim
    \frac{1}{(\rho|h_l|)^{n-\alpha}}\int_{Q_1} b(y)\phi(y)dy
    \\
    = &
    \frac{1}{(\rho|h_l|)^{n-\alpha}}\int_{Q_1} |b(y)|dy
    =
    \frac{\mu(Q_1)B}{(\rho|h_l|)^{n-\alpha}},\quad\forall\,x\in Q_{(l)},
  \end{split}
\end{equation*}
which implies that
\begin{equation}\label{for proof, structure, 2}
  \begin{split}
    \|T(b\phi)\chi_{Q_{M}}\|_{\widetilde{Y}}/\|\chi_{Q_{M}}\|_Z
    \gtrsim &
    \frac{\mu(Q_1)B}{(\rho|h_l|)^{n-\alpha}}\cdot \frac{\|\chi_{Q_{M}}\|_{\widetilde{Y}}}{\|\chi_{Q_{M}}\|_Z}.
  \end{split}
\end{equation}

On the other hand, the assumption $Y\cdot Z\subset \widetilde{Y}$ implies that
\begin{equation}\label{for proof, structure, 3}
  \begin{split}
    \frac{\|bT(\phi)\chi_{Q_{M}}\|_{\widetilde{Y}}}{\|\chi_{Q_{M}}\|_Z}
    \lesssim
    \frac{\|b\chi_{Q_{M}}\|_{Y}\|T(\phi)\chi_{Q_{M}}\|_{Z}}{\|\chi_{Q_{M}}\|_Z}.
  \end{split}
\end{equation}
The combination of (\ref{for proof, structure, 1}),(\ref{for proof, structure, 2}) and (\ref{for proof, structure, 3})
then yields that
\begin{equation}\label{for proof, structure, 8}
  \|[b,T]\phi\|_Y
  \geqslant
  \frac{A_1\mu(Q_1)B}{(\rho|h_l|)^{n-\alpha}}\cdot \frac{\|\chi_{Q_{M}}\|_{\widetilde{Y}}}{\|\chi_{Q_{M}}\|_Z}
  -\frac{A_2\|b\chi_{Q_{M}}\|_Y\|T(\phi)\chi_{Q_{M}}\|_Z}{\|\chi_{Q_{M}}\|_Z}.
\end{equation}
  Take $$\psi=\chi_{Q_1}.$$
Then there exist constants $A_3$ and $A_4$, which may be changed line to line, such that
\begin{equation}\label{for proof, structure, 5}
  \|[b,T]\psi\|_{Y}
  \geqslant
  \|[b,T]\psi\chi_{Q_{M}}\|_{Y}\geqslant A_3\|bT(\psi)\chi_{Q_{M}}\|_{Y}-A_4\|T(b\psi)\chi_{Q_{M}}\|_{Y}.
\end{equation}
By (\ref{for proof, structure, 4}), we have
\begin{equation*}
  \begin{split}
    |T(b\psi)(x)|\leqslant \int_{Q_1} \frac{|\Omega(x-y)|}{|x-y|^{n-\alpha}}|b(y)|dy
    \lesssim
    \frac{1}{(\rho|h_l|)^{n-\alpha}}\int_{Q_1} |b(y)|dy
    =
    \frac{\mu(Q_1)B}{(\rho|h_l|)^{n-\alpha}},\quad\forall\, x\in Q_{(l)}.
  \end{split}
\end{equation*}
Consequently,
\begin{equation}\label{for proof, structure, 6}
  \begin{split}
    \|T(b\psi)\chi_{Q_{M}}\|_{Y}
    \lesssim
    \frac{\mu(Q_1)B\|\chi_{Q_{M}}\|_Y}{(\rho|h_l|)^{n-\alpha}}.
  \end{split}
\end{equation}
Also, for $x\in Q_{(l)}$,
\begin{equation*}
  \begin{split}
    |b(x)T(\psi)(x)|
    =
    \left|b(x)\int_{Q_1}\frac{\Omega(x-y)}{|x-y|^{n-\alpha}}dy\right|
    \gtrsim
    \frac{|b(x)|\cdot |Q_1|}{(\rho|h_l|)^{n-\alpha}}=\frac{\rho^{\alpha}|b(x)|}{|h_l|^{n-\alpha}}.
  \end{split}
\end{equation*}
Thus,
\begin{equation}\label{for proof, structure, 7}
  \begin{split}
    \|bT(\psi)\chi_{Q_{M}}\|_{Y}
    \gtrsim
    \frac{\rho^{\alpha}\|b\chi_{Q_{M}}\|_Y}{|h_l|^{n-\alpha}}.
  \end{split}
\end{equation}
The combination of (\ref{for proof, structure, 5}), (\ref{for proof, structure, 6}) and (\ref{for proof, structure, 7}) yields that
\begin{equation}\label{for proof, structure, 9}
  \|[b,T]\psi\|_Y
  \geqslant
  \frac{A_3\rho^{\alpha}\|b\chi_{Q_{M}}\|_Y}{|h_l|^{n-\alpha}}-\frac{A_4\mu(Q_1)B\|\chi_{Q_{M}}\|_Y}{(\rho|h_l|)^{n-\alpha}}.
\end{equation}
Set
\begin{equation}
  \Xi_M : =\frac{A_2|h_l|^{n-\alpha}\|T(\phi)\chi_{Q_{M}}\|_Z}{A_3\rho^{\alpha}\|\chi_{Q_{M}}\|_Z}.
\end{equation}
By (\ref{for proof, structure, 8}), (\ref{for proof, structure, 9}), and the boundedness of $[b,T]$, we obtain that
\begin{equation*}
  \begin{split}
    (1+\Xi_M)\|[b,T]\|_{X\rightarrow Y}(\|\phi\|_X+\|\psi\|_X)
    \geqslant &
    \|[b,T]\phi\|_Y+\Xi_M\|[b,T]\psi\|_Y
    \\
    \geqslant &
    \left(\frac{A_1\|\chi_{Q_{M}}\|_{\widetilde{Y}}}{\|\chi_{Q_{M}}\|_{Z}}
    -\Xi_M A_4\|\chi_{Q_{M}}\|_{Y}\right)\frac{\mu(Q_1)B}{(\rho|h_l|)^{n-\alpha}}.
  \end{split}
\end{equation*}
Letting $M\rightarrow \infty$, we have
\begin{equation*}
  \begin{split}
    (1+\Xi_{\infty})\|[b,T]\|_{X\rightarrow Y}(\|\phi\|_X+\|\psi\|_X)
    \geqslant &
    \left(\frac{A_1\|\chi_{Q_{(l)}}\|_{\widetilde{Y}}}{\|\chi_{Q_l}\|_{Z}}
    -\Xi_{\infty}A_4\|\chi_{Q_{(l)}}\|_{Y}\right)\frac{\mu(Q_1)B}{(\rho|h_l|)^{n-\alpha}}
    \\
    = &
    (A_1-\Xi_{\infty}A_4)\frac{\mu(Q_1)B\|\chi_{Q_{(l)}}\|_{Y}}{(\rho|h_l|)^{n-\alpha}},
  \end{split}
\end{equation*}
where
\begin{equation}
  \Xi_{\infty} : =\frac{A_2|h_l|^{n-\alpha}\|T(\phi)\chi_{Q_{(l)}}\|_Z}{A_3\rho^{\alpha}\|\chi_{Q_{(l)}}\|_Z}.
\end{equation}
If we can make $\Xi_{\infty}\leqslant A_1/(2A_4)$, then
\begin{equation*}
  \begin{split}
    \frac{A_1}{2}\,\frac{\mu(Q_{1})B\|\chi_{Q_{(l)}}\|_Y}{(\rho|h_l|)^{n-\alpha}}
    \leqslant &
    (1+\Xi_{\infty})\|[b,T]\|_{X\rightarrow Y}(\|\phi\|_X+\|\psi\|_X)
    \\
    \leqslant &
    3(1+\Xi_{\infty})\|[b,T]\|_{X\rightarrow Y}\|\chi_{Q_1}\|_X.
  \end{split}
\end{equation*}
Observing that $Q_1\subset \lambda Q_{(l)}=2(|y_0-x_0|+\sqrt{n})Q_{(l)}=2(|h_l|+\sqrt{n})Q_{(l)}$, we have
\begin{equation*}
  \|\chi_{Q_1}\|_Y\leqslant \|\chi_{\lambda Q_{(l)}}\|_Y\leqslant C_{|h_l|}\|\chi_{Q_{(l)}}\|_Y.
\end{equation*}
In what follows, we will prove the desired result only under the condition $(a)$, since the arguments under the condition $(b)$ is similar.
Using the assumption $\|\chi_{Q_1}\|_X\lesssim \|\chi_{Q_1}\|_Y\mu(Q_1)|Q_1|^{\alpha/n-1}$,
we then deduce that
\begin{equation*}
  \|\chi_{Q_1}\|_X\lesssim \|\chi_{Q_1}\|_Y\mu(Q_1)|Q_1|^{\alpha/n-1}
  \lesssim
  C_{|h_l|}\|\chi_{Q_{(l)}}\|_Y\mu(Q_1)|Q_1|^{\alpha/n-1}.
\end{equation*}
Consequently,
\begin{equation*}
  \begin{split}
  \frac{A_1}{2}\,\frac{\mu(Q_1)B\|\chi_{Q_{(l)}}\|_Y}{(\rho|h_l|)^{n-\alpha}}
    \leqslant &
    3(1+\Xi_{\infty})\|[b,T]\|_{X\rightarrow Y}\|\chi_{Q_1}\|_X
    \\
    \lesssim &
    3C_{|h_l|}(1+\Xi_{\infty})\|[b,T]\|_{X\rightarrow Y}\|\chi_{Q_{(l)}}\|_Y\mu(Q_1)\rho^{\alpha-n}.
  \end{split}
\end{equation*}
This implies that
\begin{equation*}
  B\lesssim \frac{6(1+\Xi_{\infty})C_{|h_l|}|h_l|^{n-\alpha}}{A_1}\cdot \|[b,T]\|_{X\rightarrow Y}.
\end{equation*}

The remaining question is how to make $\Xi_{\infty}$ small. Write
$$Q_{(l)}=\rho Q_0+\rho x_0,\qquad Q_1=\rho Q_0+\rho y_0.$$
Note that $x_0-y_0=h_l$.
For $x\in Q_{(l)}$, we have
\begin{equation*}
  \begin{split}
    |T(\phi)(x)|
    = &
    \Big|\int_{Q_{1}}\frac{\Omega(x-y)}{|x-y|^{n-\alpha}}\phi(y)dy\Big|
    =
    \Big|\int_{Q_{1}}\left(\frac{\Omega(x-y)}{|x-y|^{n-\alpha}}-\frac{\Omega(x-\rho y_0)}{|x-\rho y_0|^{n-\alpha}}\right)\phi(y)dy\Big|
    \\
    \leqslant &
    2\int_{Q_{1}}\left|\frac{\Omega(x-y)}{|x-y|^{n-\alpha}}-\frac{\Omega(x-\rho y_0)}{|x-\rho y_0|^{n-\alpha}}\right|dy
    \\
    \leqslant &
    2\int_{Q_{1}}|\Omega(x-y)|\left|\frac{1}{|x-y|^{n-\alpha}}-\frac{1}{|x-\rho y_0|^{n-\alpha}}\right|dy
    +
    2\int_{Q_{1}}\frac{|\Omega(x-y)-\Omega(x-\rho y_0)|}{|x-\rho y_0|^{n-\alpha}}dy
    \\
    \lesssim &
    \frac{\rho|Q_{1}|}{|\rho(x_0-y_0)|^{n-\alpha+1}}+\frac{1}{|\rho(x_0-y_0)|^{n-\alpha}}\int_{Q_{1}}|\Omega(x-y)-\Omega(x-\rho y_0)|dy
    \\
    = &
    \frac{\rho^{\alpha}}{|h_l|^{n-\alpha}}
    \left(\frac{1}{|h_l|}+\frac{1}{\rho^n}\int_{Q_{1}}|\Omega(x-y)-\Omega(x-\rho y_0)|dy\right).
  \end{split}
\end{equation*}
Consequently,
\begin{equation*}
  \begin{split}
    \|T(\phi)\chi_{Q_{(l)}}\|_Z
    \lesssim &
    \frac{\rho^{\alpha}}{|h_l|^{n-\alpha}}
    \left(\frac{\|\chi_{Q_{(l)}}\|_Z}{|h_l|}+\frac{1}{\rho^n}\left\|\int_{Q_{1}}|\Omega(\cdot-y)-\Omega(\cdot-\rho y_0)|dy\right\|_{Z(Q_{(l)})}\right).
  \end{split}
\end{equation*}
This shows that
\begin{equation*}
  \begin{split}
    \Xi_{\infty}
    \lesssim &
    \frac{A_2}{A_3|h_l|}+\frac{A_2}{A_3\rho^n\|\chi_{Q_{(l)}}\|_Z}\left\|\int_{Q_{1}}|\Omega(\cdot-y)-\Omega(\cdot-\rho y_0)|dy\right\|_{Z(Q_{(l)})}.
  \end{split}
\end{equation*}
And
\begin{equation*}
  \begin{split}
  \frac{1}{\rho^n\|\chi_{Q_{(l)}}\|_Z}&\left\|\int_{Q_{1}}|\Omega(\cdot-y)-\Omega(\cdot-\rho y_0)|dy\right\|_{Z(Q_{(l)})}
    \\
    = &
    \frac{1}{\|\chi_{Q_{(l)}}\|_Z}\left\|\int_{Q_0}|\Omega(\cdot-\rho y-\rho y_0)-\Omega(\cdot-\rho y_0)|dy\right\|_{Z(Q_{(l)})}
    \\
    \sim &
    \frac{1}{\|\chi_{Q_0}\|_Z}\left\|\int_{Q_0}|\Omega(\cdot-y+x_0-y_0)-\Omega(\cdot+x_0-y_0)|dy\right\|_{Z(Q_0)}
    \\
    \sim &
    \frac{1}{\|\chi_{Q_0}\|_Z}\left\|\int_{Q_0}|\Omega(\cdot-y+h_l)-\Omega(\cdot+h_l)|dy\right\|_{Z(Q_0)}.
  \end{split}
\end{equation*}
Recalling (\ref{for proof, structure, 10}) and $|h_l|\rightarrow \infty$, we obtain
\begin{equation*}
  \frac{1}{\rho^n\|\chi_{Q_{(l)}}\|_Z}\left\|\int_{Q_{1}}|\Omega(\cdot-y)-\Omega(\cdot-\rho y_0)|dy\right\|_{Z(Q_{(l)})}\rightarrow 0
  \ \text{as}\ l\rightarrow \infty,
\end{equation*}
which implies that $\Xi_{\infty} \rightarrow 0$ as $l\rightarrow \infty$, uniformly for all $\rho>0$ and $y_0\in \mathbb{R}^n$.
Take $l=\mathbf{l_0}$ such that $\Xi_{\infty}\leqslant 1$. Then
\begin{equation*}
    B\lesssim \frac{6(1+\Xi_{\infty})C_{|h_{\mathbf{l_0}}|}|h_{\mathbf{l_0}}|^{n-\alpha}}{A_1}\cdot \|[b,T]\|_{X\rightarrow Y}
     \lesssim
    C_{|h_{\mathbf{l_0}}|}|h_{\mathbf{l_0}}|^{n-\alpha}\cdot \|[b,T]\|_{X\rightarrow Y}.
\end{equation*}
This completes the proof of Proposition 2.1.
\end{proof}

Note that, in some endpoint cases, the "right" boundedness of commutator can not be regarded as
 the usual boundedness between two Quasi-Banach spaces.
For instance, different from the weak (1,1) boundedness of Calder\'on-Zygmund operator,
the commutator associated with Calder\'on-Zygmund operator is not even weak type $(1,1)$ (see \cite{Perez-1995}).
In order to deal with these situations, we present the following pointwise estimates, which will be very helpful in applications below.

\begin{proposition}[Technique, homogeneous kernel, pointwise estimates]\label{proposition, technique, homogeneous kernel, pointwise}
  Let $\Omega$ be a homogeneous function  of degree zero, and suppose that $\Omega$ satisfies the following local property in $\mathbb{S}^{n-1}$: there exists an open subset of $\mathbb{S}^{n-1}$, denoted by $E$, such that
  \begin{enumerate}[(i)]
  \item lower and upper bound: \\$c\leqslant \Omega(x') \leqslant C$ for all $x'\in E$, where $0<c<C$, or $c<C<0$.
  \item for any open subset $F\subset E$, there exists a sequence $\{h_l\}_{l=1}^{\infty}\subset\mathbb{R}^n$
  satisfying that $(h_l)':=h_l/|h_l|\in F$, $|h_l|\rightarrow \infty$ as $l\rightarrow \infty$, and
  \begin{equation*}
    \lim\limits_{\substack{l\rightarrow \infty}}
      \left\|\int_{Q_0}|\Omega(\cdot-y+h_l)-\Omega(\cdot+h_l)|dy\right\|_{L^{\infty}(Q_0)}\rightarrow 0.
  \end{equation*}
  \end{enumerate}
  Let $T_{\alpha}$ be the integral operator defined in (\ref{integral operator, linear case})
  and $[b,T_{\alpha}]$ be the corresponding commutator associated with $b$ and $T_{\alpha}$.
  Then, for any cube $Q_1\subset \mathbb{R}^n$, there exist two functions $\phi$ and $\psi$ satisfying
  $|\phi|,|\psi|\leq 2\chi_{Q_1}$, and a cube $Q$ with the same side length of $Q_1$, such that
  $Q_1\subset \lambda Q$ for some $\lambda>0$ independent of $Q_1$, and
  \begin{equation*}
 \frac{1}{|Q_1|^{1-\alpha/n}}\int_{Q_1}|b(y)-b_{Q_1}|dy
 \leq \widetilde{C}
 \left(|[b,T](\phi)(x)|+|[b,T](\psi)(x)|\right)\ \text{for\ all\ }x\in Q,
\end{equation*}
where the constant $\widetilde{C}$ is independent of the choice of $Q_1$.
\end{proposition}

\begin{proof} Without loss of generality, we may assume $0<c<C$.
For simplicity in notation and proof, $T_{\alpha}$ will be abbreviated to $T$.
By the lower and upper bound of $\Omega$ and the homogeneous condition of degree zero, we can find a constant $\tau_0$ and a nonempty open cone $\Gamma\subset \mathbb{R}^n$ with vertex at the origin,
such that for any $u\in 2Q_0$, $v\in \Gamma_{\tau_0}=\Gamma\cap B^c(0,\tau_0)$, we have $(u+v)':=(u+v)/|u+v|\in E$,
  \begin{equation}
    \Omega(u+v)\sim 1, \ |u+v|\sim |v|,
  \end{equation}
Furthermore, we can find a sequence $\{h_l\}_{l=1}^{\infty}$ satisfying that $h_l\in \Gamma_{\tau_0}$, and $|h_l|\rightarrow \infty$
as $l\rightarrow \infty$, such that
\begin{equation}\label{for proof, pointwise, 1}
  \lim\limits_{{l\rightarrow \infty}}
      \left\|\int_{Q_0}|\Omega(\cdot-y+h_l)-\Omega(\cdot+h_l)|dy\right\|_{L^{\infty}(Q_0)}\rightarrow 0.
\end{equation}
For a fixed cube $Q_1$ with side length $\rho>0$, without loss of generality, we may assume
\begin{equation*}
  \int_{Q_1}b(y)dy=0,
\end{equation*}
and set
\begin{equation}
  B:=\frac{1}{|Q_1|}\int_{Q_1}|b(y)|dy.
\end{equation}
Take
\begin{equation*}
  \phi(x)=\Big(sgn(b(x))-\frac{1}{|Q_1|}\int_{Q_1}sgn(b(y))dy\big)\chi_{Q_1}(x).
\end{equation*}
Then $|\phi|\leq 2\chi_{Q_1}$, $b\phi\geqslant 0$.
  Take $Q_{(l)}: =Q_1+\rho h_l$.
  As in the proof of Theorem \ref{theorem, structure, homogeneous kernel}, for $x\in Q_{(l)}$ we have
\begin{equation*}
  \begin{split}
    |T(b\phi)(x)|=&\int_{Q_1} \frac{\Omega(x-y)}{|x-y|^{n-\alpha}}b(y)\phi(y)dy
    \gtrsim
    \frac{1}{(\rho|h_l|)^{n-\alpha}}\int_{Q_1} b(y)\phi(y)dy
    \\
    = &
    \frac{1}{(\rho|h_l|)^{n-\alpha}}\int_{Q_1} |b(y)|dy
    =
    \frac{|Q_1|B}{(\rho|h_l|)^{n-\alpha}}.
  \end{split}
\end{equation*}
This implies that for $x\in Q_{(l)}$,
\begin{equation}\label{for proof, pointwise, 2}
    \begin{split}
      |[b,T](\phi)(x)|
      \geqslant &
      |T(b\phi)(x)|-|b(x)T(\phi)(x)|
      \\
      \geqslant &
      \frac{A_1|Q_1|B}{(\rho|h_l|)^{n-\alpha}}-|b(x)T(\phi)(x)|
            \\
      = &
      \frac{A_1B\rho^{\alpha}}{|h_l|^{n-\alpha}}-|b(x)T(\phi)(x)|.
    \end{split}
  \end{equation}
On the other hand, we take $\psi=\chi_{Q_1}$
and deduce that
\begin{equation*}
  \begin{split}
    |T(b\psi)(x)|\leqslant \int_{Q_1} \frac{|\Omega(x-y)|}{|x-y|^{n-\alpha}}|b(y)|dy
    \lesssim &
    \frac{1}{(\rho|h_l|)^{n-\alpha}}\int_{Q_1} |b(y)|dy
    =
    \frac{B\rho^{\alpha}}{|h_l|^{n-\alpha}},\quad\forall\,x\in Q_{(l)},
  \end{split}
\end{equation*}
and
\begin{equation*}
  \begin{split}
    |b(x)T(\psi)(x)|
    =&
    \left|b(x)\int_{Q_1}\frac{\Omega(x-y)}{|x-y|^{n-\alpha}}dy\right|
    \\
    \gtrsim &
    \frac{|b(x)|\cdot |Q_1|}{(\rho|h_l|)^{n-\alpha}}=\frac{|b(x)|\rho^{\alpha}}{|h_l|^{n-\alpha}},\quad\forall\, x\in Q_{(l)}.
  \end{split}
\end{equation*}
Consequently,
\begin{equation}\label{for proof, pointwise, 3}
  \begin{split}
    |[b,T](\psi)(x)|
      \geqslant
      |b(x)T(\psi)(x)|-|T(b\psi)(x)|
          \geqslant
      \frac{A_3|b(x)|\rho^{\alpha}}{|h_l|^{n-\alpha}}-\frac{A_4B\rho^{\alpha}}{|h_l|^{n-\alpha}},\quad\forall\, x\in Q_{(l)}.
  \end{split}
\end{equation}
Denote
\begin{equation*}
  \Xi(x):=\frac{h_l^{n-\alpha}T(\phi)(x)}{A_3\rho^{\alpha}}.
\end{equation*}
Using (\ref{for proof, pointwise, 1}) and the same arguments used in the proof of Theorem \ref{theorem, structure, homogeneous kernel},
we have $\Xi(x) \rightarrow 0$ uniformly for $Q_1$ and all $x\in Q_{(l)}$, as $l\rightarrow \infty$.
Take sufficient large $l=\mathbf{l_0}$ independent of $Q_1$,
such that
\begin{equation*}
  \Xi(x)\leqslant \min\{\frac{A_1}{2A_4},1\}, \ \forall x\in Q_{(\mathbf{l_0})}.
\end{equation*}
Then the combination of (\ref{for proof, pointwise, 2}) and (\ref{for proof, pointwise, 3}) yields that
\begin{equation*}
 |[b,T](\phi)(x)|+|[b,T](\psi)(x)|\geqslant|[b,T](\phi)(x)|+\Xi|[b,T](\psi)(x)|
 \geqslant \frac{A_1B|Q_1|^{\alpha/n}}{2|h_{\mathbf{l_0}}|^{n-\alpha}},\quad\forall\, x\in Q_{(\mathbf{l_0})}=Q_1+\rho h_{\mathbf{l_0}}.
\end{equation*}
Take $Q=Q_{(\mathbf{l_0})}$, and recall the side length of $Q_1$ is $\rho$.
Thus, there exists $\lambda$ depend only on $h_{\mathbf{l_0}}$,
such that $Q_1\subset \lambda Q$. We have now completed this proof.
\end{proof}

\begin{proposition}[Technique, homogeneous kernel, $Z=L^{\infty}$]\label{proposition, technique, homogeneous kernel, Z=L^infty}
Let $\Omega$ be a homogeneous function of degree zero.
  If $\Omega$ satisfies the following local uniform Lebesgue point property in $\mathbb{S}^{n-1}$:
  there exists an open subset of $\mathbb{S}^{n-1}$, denoted by $E$, such that
  \begin{equation*}
    \lim_{r\rightarrow 0}\frac{1}{\sigma(B(x',r)\cap \mathbb{S}^{n-1})}\int_{B(x',r)\cap \mathbb{S}^{n-1}}|\Omega(z')-\Omega(x')|d\sigma(z')=0
  \end{equation*}
  for all $x'\in E$ uniformly, then, for any $F\subset E$ with $\overline{F}\cap E^c=\emptyset$, we have
  \begin{equation*}
    \lim\limits_{\substack{|h|\rightarrow \infty, \\h'\in F}}
     \left\|\int_{Q_0}|\Omega(\cdot-y+h)-\Omega(\cdot+h)|dy\right\|_{L^{\infty}(Q_0)}\rightarrow 0.
  \end{equation*}
\end{proposition}

\begin{proof}
By the assumption, we have $(x+h)':=(x+h)/|x+h|\in E$ for $x\in Q_0$ and sufficient large $h'\in F$.
For any $x\in Q_0$, we write $x_h:=x+h$. Then
\begin{equation}
    \begin{split}
      \int_{Q_0}\left|\Omega(x_h-y)-\Omega(x_h)\right|dy
      =&
      \int_{x_h+Q_0}\left|\Omega(z)-\Omega(x_h)\right|dz
      \\
      \leqslant &
      \int_{|x_h|-\sqrt{n}}^{|x_h|+\sqrt{n}}\int_{\mathbb{S}^{n-1}\cap B(x_h',\sqrt{n}/|x_h|)}|\Omega(z')-\Omega(x_h')|d\sigma(z')r^{n-1}dr,
    \end{split}
  \end{equation}
  where we use the fact that for $x,y \in Q_0$,
  \begin{equation*}
    \begin{split}
      |z'-x_h'|
      =\left|\frac{x_h-y}{|x_h-y|}-\frac{x_h}{|x_h|}\right|
      \leqslant
      |x_h-y|\cdot\left|\frac{1}{|x_h-y|}-\frac{1}{|x_h|}\right|+\frac{|y|}{|x_h|}\leqslant \frac{2|y|}{|x_h|}\leqslant \frac{\sqrt{n}}{|x_h|}.
    \end{split}
  \end{equation*}
  Thus,
  \begin{equation*}
    \begin{split}
      \int_{Q_0}\left|\Omega(x_h-y)-\Omega(x_h)\right|dy
      \leqslant&
      \int_{|x_h|-\sqrt{n}}^{|x_h|+\sqrt{n}}\int_{\mathbb{S}^{n-1}\cap B(x_h',\sqrt{n}/|x_h|)}|\Omega(z')-\Omega(x_h')|d\sigma(z')r^{n-1}dr
      \\
      \sim &
      |x_h|^{n-1}\int_{\mathbb{S}^{n-1}\cap B(x_h',\sqrt{n}/|x_h|)}|\Omega(z')-\Omega(x_h')|d\sigma(z')
      \\
      \sim &
      \frac{1}{\sigma(\mathbb{S}^{n-1}\cap B(x_h',\sqrt{n}/|x_h|))}\int_{\mathbb{S}^{n-1}\cap B(x_h',\sqrt{n}/|x_h|)}|\Omega(z')-\Omega(x_h')|d\sigma(z').
    \end{split}
  \end{equation*}
  Recalling $x_h'\in E$, $|x_h|\rightarrow \infty$ as $|h|\rightarrow \infty$ for $x\in Q_0$, and recalling that $E$ is the set of uniform Lebesgue points, we conclude that
  \begin{equation*}
    \lim_{x\in Q_0, |h|\rightarrow \infty} \frac{1}{\sigma(\mathbb{S}^{n-1}\cap B(x_h',\sqrt{n}/|x_h|))}\int_{\mathbb{S}^{n-1}\cap B(x_h',\sqrt{n}/|x_h|)}|\Omega(z')-\Omega(x_h')|d\sigma(z')=0.
  \end{equation*}
 This completes the proof of Proposition 2.3.
\end{proof}

\begin{proposition}[Technique, homogeneous kernel, $Z=L^1$]\label{proposition, technique, homogeneous kernel, Z=L^1}
Let $\Omega$ be a function satisfying homogeneous condition of degree zero in $\mathbb{R}^n$.
  If $\Omega\in L^1(E)$ for some open subset $E$ of $\mathbb{S}^{n-1}$,
  then for any $F\subset E$ with $\overline{F}\cap E^c=\emptyset$, we have
  \begin{equation*}
    \lim\limits_{\substack{d\rightarrow \infty}}
      \inf\limits_{\substack{h'\in F\\|h|=d}}
     \left\|\int_{Q_0}|\Omega(\cdot-y+h)-\Omega(\cdot+h)|dy\right\|_{L^1(Q_0)}\rightarrow 0.
  \end{equation*}
\end{proposition}
\begin{proof}
For a sufficient large $d>0$, there exists an index set $\bigwedge_d$ and
a sequence $\{h_{d,j}\}_{j\in \bigwedge_d}\subset \mathbb{R}^n$, such that
$h_{d,j}'\in F$, $|h_{d,j}|=d$ for all $j\in \bigwedge_d$, $|\bigwedge_d|\sim d^{n-1}$,
$\sum_{j\in \bigwedge_d}\chi_{B(h'_{d,j},\sqrt{n}/d)\cap \mathbb{S}^{n-1}}\lesssim \chi_{F}(x)$.
Thus,
  \begin{equation*}
    \begin{split}
      &\sum_{j\in \bigwedge_d}\left\|\int_{Q_0}|\Omega(\cdot-y+h_{d,j})-\Omega(\cdot+h_{d,j})|dy\right\|_{L^1(Q_0)}
      \\
      &\qquad\qquad =
      \sum_{j\in \bigwedge_d}\int_{Q_0}\int_{Q_0}|\Omega(x-y+h_{d,j})-\Omega(x+h_{d,j})|dxdy
      \\
      &\qquad\qquad=
      \sum_{j\in \bigwedge_d}\int_{Q_0}\int_{Q_0+h_{d,j}}|\Omega(x-y)-\Omega(x)|dxdy
      \\
      &\qquad\qquad\leqslant
      \sum_{j\in \bigwedge_d}\int_{Q_0}\int_{d-\sqrt{n}}^{d+\sqrt{n}}\int_{\mathbb{S}^{n-1}\cap B(h_{d,j}',\sqrt{n}/d)}|\Omega(x'-y/r)-\Omega(x')|d\sigma(x')r^{n-1}drdy
      \\
      &\qquad\qquad=
      \int_{Q_0}\int_{d-\sqrt{n}}^{d+\sqrt{n}}\sum_{j\in \bigwedge_d}\int_{\mathbb{S}^{n-1}\cap B(h_{d,j}',\sqrt{n}/d)}|\Omega(x'-y/r)-\Omega(x')|d\sigma(x')r^{n-1}drdy,
    \end{split}
  \end{equation*}
  where we use the fact that for $x\in Q_0+h_{d,j}$,
  \begin{equation*}
    \begin{split}
      |x'-h_{d,j}'|=&\left|\frac{x}{|x|}-\frac{h_{d,j}}{|h_{d,j}|}\right|
      \leqslant
      \left|\frac{x-h_{d,j}}{|h_{d,j}|}\right|+|x|\cdot\left|\frac{1}{|x|}-\frac{1}{|h_{d,j}|}\right|
      \leqslant
      \frac{2|x-h_{d,j}|}{|h_{d,j}|}
      \leqslant \frac{\sqrt{n}}{|h_{d,j}|}=\frac{\sqrt{n}}{d}.
    \end{split}
  \end{equation*}
  Recalling the choice of $h_{d,j}$, we have
  \begin{equation*}
    \begin{split}
      \sum_{j\in \bigwedge_d}\int_{\mathbb{S}^{n-1}\cap B(h_{d,j}',\sqrt{n}/d)}|\Omega(x'-y/r)-\Omega(x')|d\sigma(x')
      \lesssim
      \int_{F}|\Omega(x'-y/r)-\Omega(x')|d\sigma(x').
    \end{split}
  \end{equation*}
Thus,
\begin{equation*}
    \begin{split}
     &|\bigwedge_d|\cdot \inf\limits_{\substack{h'\in F\\|h|=d}}\left\|\int_{Q_0}|\Omega(\cdot-y+h)-\Omega(\cdot+h)|dy\right\|_{L^1(Q_0)}
      \\
      &\qquad\qquad\leq
      \sum_{j\in \bigwedge_d}\left\|\int_{Q_0}|\Omega(\cdot-y+h_{d,j})-\Omega(\cdot+h_{d,j})|dy\right\|_{L^1(Q_0)}
      \\
      &\qquad\qquad\lesssim
      \int_{Q_0}\int_{d-\sqrt{n}}^{d+\sqrt{n}}\int_{F}|\Omega(x'-y/r)-\Omega(x')|d\sigma(x')r^{n-1}drdy
      \\
      &\qquad\qquad\lesssim
      d^{n-1}\max_{|z|\leqslant \sqrt{n}/d} \int_{F}|\Omega(x'-z)-\Omega(x')|d\sigma(x').
    \end{split}
  \end{equation*}
Recalling $|\bigwedge_d|\sim d^{n-1}$, we obtain
\begin{equation*}
    \begin{split}
\inf\limits_{\substack{h'\in F\\|h|=d}}\left\|\int_{Q_0}|\Omega(\cdot-y+h)-\Omega(\cdot+h)|dy\right\|_{L^1(Q_0)}
      &\lesssim  \max_{|z|\leqslant \sqrt{n}/d} \int_{F}|\Omega(x'-z)-\Omega(x')|d\sigma(x')\\
      &\lesssim  \max_{\|\rho\|\leqslant \sqrt{n}/d} \int_{F}|\Omega(\rho x')-\Omega(x')|d\sigma(x')\\
      &\lesssim  \max_{\|\rho\|\leqslant \sqrt{n}/d} \int_{\mathbb{S}^{n-1}}|\Omega\chi_E(\rho x')-\Omega\chi_E(x')|d\sigma(x')\rightarrow 0
    \end{split}
  \end{equation*}
as $d\rightarrow \infty$.
Here $\rho$ is a rotation on $\mathbb{R}^n$, $\|\rho\|=\sup\{|\rho x'-x'|: x'\in \mathbb{S}^{n-1}\}$.
Proposition 2.4 is proved.
\end{proof}

The combination of Theorem \ref{theorem, structure, homogeneous kernel} and Proposition \ref{proposition, technique, homogeneous kernel, Z=L^1}
yields the following useful conclusion.

\begin{proposition}\label{proposition, structure and technique for Z=L^1}
  Let $Z=L^1(\mathbb{R}^n)$,
  $X$, $Y$, $\widetilde{Y}$ be Quasi-Banach spaces, satisfying all
  the assumptions as described in Theorem \ref{theorem, structure, homogeneous kernel}.
  Let $T_{\alpha}$ be the integral operator associated with $\Omega$ and $\alpha$,
  where $\Omega$ is a homogeneous function of degree zero. Suppose $b\in L_{loc}^1(\mathbb{R}^n)$,
  $[b,T_{\alpha}]$ is bounded from $X$ to $Y$. If $\Omega$ satisfies the following local lower and upper bound property: there
  exists a nonempty open subset $E\subset \mathbb{S}^{n-1}$ such that
  \begin{equation*}
    c\leqslant \Omega(x')\leqslant C\ \text{for all}\ x'\in E,\text{where}\  0<c<C,\ \text{or}\  c<C<0,
  \end{equation*}
  then $b\in BMO_\mu$, and $\|b\|_{BMO_\mu}\lesssim \|[b, T_{\alpha}]\|_{X\rightarrow Y}.$
\end{proposition}
\begin{proof}
  By the assumption, $\Omega\in L^1(E)$.
  For any open set $F\subset E$ satisfying $\overline{F}\cap E^c=\emptyset$, we use
  Proposition \ref{proposition, technique, homogeneous kernel, Z=L^1} to deduce that
  \begin{equation*}
    \lim\limits_{\substack{|h|\rightarrow \infty}}
      \inf_{h'\in F}
     \left\|\int_{Q_0}|\Omega(\cdot-y+h)-\Omega(\cdot+h)|dy\right\|_{L^1(Q_0)}\rightarrow 0.
  \end{equation*}
  Then the desired conclusion immediately follows from Theorem \ref{theorem, structure, homogeneous kernel}.
\end{proof}

\section{The general structure and technique theorems in multi-linear setting}

This section is devoted to the investigation of the multilinear commutators. At first, we recall some notations and definitions. Let $m\in\mathbb{N}$, $K_{\alpha}$ be a function defined off the diagonal $x=y_1=\cdots=y_m$ in $(\mathbb{R}^n)^{m+1}=:\mathbb{R}^{(m+1)n}$ satisfying
\begin{equation*}
  |K_{\alpha}(x,y_1,\cdots,y_m)|\lesssim \frac{1}{(\sum_{j=1}^m|x-y_j|)^{mn-\alpha}}
\end{equation*}
and
\begin{equation*}
\begin{split}
  |K_{\alpha}(x;y_1,\cdots,y_i,\cdots,y_m)-K_{\alpha}(x;y_1,\cdots,y_i',\cdots,y_m)|
\lesssim \frac{|y_i-y_i'|^{\delta}}{(\sum_{j=1}^m|x-y_j|)^{mn-\alpha+\delta}}
\end{split}
\end{equation*}
for some $\delta>0$, and all $i=1,\cdots,m$, whenever $|y_i-y_i'|\leq \frac{1}{2}\max_{1\leq j\leq m}\{|x-y_j|\}$.

For $\vec{f}=(f_1,\cdots,f_m)$, we consider the following $m$-linear operator $T_{K_\alpha}$ associated with kernel $K_{\alpha}$, defined by
\begin{equation}
  T_{K_{\alpha}}(\vec{f})(x):=\int_{\mathbb{R}^{mn}}K_{\alpha}(x;y_1,\cdots,y_m)f_1(y_1)\cdots f_m(y_m)d\vec{y}
\end{equation}
where above equality holds for all $x\notin \bigcap_{j=1}^m\text{supp}f_j$, $d\vec{y}=dy_1\cdots\,dy_m$.

\medskip

For $b\in L_{\rm loc}^1(\mathbb{R}^n)$, $1\le i\le m$, the $i$-th commutator associated $b$ and $T_{K_\alpha}$ is defined by
\begin{equation}
[b,T_{K_\alpha}]_i(\vec{f}):=bT_{K_\alpha}(f_1,\cdots,f_i,\cdots, f_m)-T_{K_\alpha}(f_1,\cdots, bf_i,\cdots,f_m).
\end{equation}

 When $\alpha=0$, $T_{K_0}$ is the $m$-linear Calder\'{o}n-Zygmund operator.

 When $\alpha\in (0,n)$, $T_{K_{\alpha}}$ is controlled by the following $m$-linear fractional integral operator:
\begin{equation*}
  I_{\alpha,m}(\vec{f})(x):=\int_{\mathbb{R}^{mn}}\frac{f_1(y_1)\cdots f_m(y_m)}{|(x-y_1,\cdots,x-y_m)|^{mn-\alpha}}d\vec{y};
\end{equation*}
and the corresponding $i$-th commutator of $I_{\alpha,m}$ is defined by
\begin{equation}
\begin{array}{ll}
[b, I_{\alpha,m}]_i(\vec{f})(x):&=b(x)I_{\alpha,m}(f_1,\cdots,f_i,\cdots,f_m)(x)-I_{\alpha,m}(f_1,\cdots,bf_i,\cdots,f_m)(x)\\
&=\displaystyle\int_{\mathbb{R}^{mn}}\frac{f_1(y_1)\cdots\,[b(x)-b(y_i)]f_i(y_i)\cdots\,f_m(y_m)}{|(x-y_1,\cdots,x-y_i,\cdots,x-y_m)|^{mn-\alpha}}d\vec{y}.
\end{array}
\end{equation}

In \cite{Chaf-Cruz-2017}, Chaffee and Cruz-Uribe recently established the following result.

\medskip

{\bf Theorem B} (\cite{Chaf-Cruz-2017}).\quad{\it Let $m,\,n\in \mathbb{N}$. Given Banach function spaces $X_1,\cdots,X_m$ and $Y$, $0\le \alpha<mn$, suppose that for every cube $Q$,
$$|Q|^{-\alpha/n}\|\chi_Q\|_{Y'}\|\chi_Q\prod_{j=1}^m\|_{X_j}\lesssim |Q|.$$
Let $T$ be a $m$-linear operator defined on $X_1\times\cdots\times X_m$, which can be represented by
$$T(\vec{f})(x)=\int _{\mathbb{R}^{mn}}K(x-y_1,\cdots,x-y_m)f_1(y_1)\cdots f_m(y_m)d\vec{y}$$
for all $x\notin \bigcap_{j=1}^m{\rm supp}(f_j)$, where $\vec{f}=(f_1,\cdots,f_m)$, $d\vec{y}=dy_1\cdots dy_m$, $K$ is a homogeneous kernel of degree $-mn+\alpha$. Suppose further that there exists a ball $B\subset \mathbb{R}^{mn}$ on which $1/K$ has an absolutely convergent Fourier series. If for some $i\in \{1,\cdots,m\}$, the $i$-th commutator $[b,T]_i$ is bounded from $X_1\times\cdots\times X_m$ to $Y$, then $b\in BMO(\mathbb{R}^n)$.}

\medskip

{\bf Remark 3.1.}\quad We remark that the methods employed in \cite{Chaf-Cruz-2017} do not work if replacing Banach spaces by Quasi-Banach spaces in Theorem B.

\medskip

Our next theorem will relax the restriction of Banach function spaces to quasi-Banach spaces and extend $BMO(\mathbb{R}^n)$ to the general $BMO_\mu$, which includes $BMO(\mathbb{R}^n)$, ${\rm Lip}_\beta (\mathbb{R}^n)$ and their weighted versions. Moreover, the condition, which the kernel $K$ satisfies, will be weakened essentially.

\begin{theorem}[Structure, standard kernel, multilinear case]\label{Theorem, structure, standard kernel}
Let $m\in \mathbb{N}$, $0\le\alpha<mn$, and $i$ be a given integer with $1\le i\le m$. Let $X_j\  (j=1,\cdots,m), Y$ be Quasi-Banach spaces, and $\mu$ be a positive function defined on all cubes of $\mathbb{R}^n$,
satisfying
\begin{equation}
  \prod_{j=1}^m\|\chi_{Q}\|_{X_j}\lesssim \|\chi_{Q}\|_Y|Q|^{\alpha/n-1}\mu(Q) \text{ for all cubes}\ Q.
\end{equation}
and one of the following conditions:
\begin{enumerate}[(a)]
  \item $\|\chi_{\lambda Q}\|_{X_j}\leq C_{\lambda}\|\chi_Q\|_{X_j} (j=1,\cdots,m)$,
  $\|\chi_{\lambda Q}\|_{Y}\leq C_{\lambda}\|\chi_Q\|_{Y}$,\
    for  $\lambda>1$\  and all cubes $Q\subset\mathbb{R}^n$;
  \item $\|\chi_{\lambda Q}\|_{X_j}\leq C_{\lambda}\|\chi_Q\|_{X_j} (j=1,\cdots,m)$, $\mu(\lambda Q)\leq C_{\lambda}\mu(Q)$,\
    for  $\lambda>1$\  and all cubes $Q\subset\mathbb{R}^n$.
\end{enumerate}

Suppose that the following conditions holds:
\begin{equation*}
  \|\chi_{\lambda Q}\|_{X_j}\leq C_{\lambda}\|\chi_Q\|_{X_j} (j=1,\cdots,m), \mu(\lambda Q)\leq C_{\lambda}\mu(Q),\
    \text{for}\  \lambda>1\  \text{and all cubes}\   Q\subset\mathbb{R}^n.
\end{equation*}
  If $[b,T_{K_{\alpha}}]_i$ is a bounded operator from $X_1\times\cdots \times X_m$ to $Y$, and there exist two Quasi-Banach spaces $\widetilde{Y}$ and $Z$ satisfying
   $$Y\cdot Z\subset\widetilde{Y},\,\,\|\chi_{Q}\|_{\widetilde{Y}}\sim \|\chi_{Q}\|_{Y}\cdot \|\chi_{Q}\|_{Z}\quad\text{for all cubes}\,\, Q\subset\mathbb{R}^n,$$
and $$\frac{\|f(a\cdot+b)\|_Z}{\|\chi_{Q_0}(a\cdot+b)\|_Z} \sim \frac{\|f(\cdot)\|_Z}{\|\chi_{Q_0}(\cdot)\|_Z}\quad\text{for all}\,f\in Z,\, \text{where}\,\, Q_0=[-1/2,1/2]^n$$
  such that $K_{\alpha}$ satisfies the following local properties: there exists an open cone $\widetilde{\Gamma}$ of $(\mathbb{R}^{n})^m$
  whose vertex is $0$,
  such that :
  \begin{enumerate}[(1)]
  \item lower and upper bound: $$\frac{c}{(\sum_{j=1}^m|x-y_j|)^{mn-\alpha}}\leqslant K_{\alpha}(x,y_1,\cdots,y_m) \leqslant \frac{C}{(\sum_{j=1}^m|x-y_j|)^{mn-\alpha}}$$
  for all $(x-y_1,\cdots,x-y_m)\in \widetilde{\Gamma}$ with $0<c<C$, or $c<C<0$;
  \item
  for any open cone $\overline{\Gamma}\subset \widetilde{\Gamma}$,
  there exists a sequence $\{h_l=(h_l^1,\cdots,h_l^m)\}_{j=1}^{\infty}$
  satisfying that $h_l\in \overline{\Gamma}$, $|h_l|\rightarrow \infty$ as $l\rightarrow \infty$, and
  \begin{equation*}
    \begin{split}
      &\left\|\int_{\prod\limits_{{j=1}}^m(Q-\sqrt[n]{Q}h_l^j)}|K_{\alpha}(\cdot,\cdots,y_m)
      -K_{\alpha}(\cdot,y_1,\cdots,y_{i-1},c_{Q}-\sqrt[n]{|Q|}h_l^i,\cdots,y_m)|dy_1\cdots\,dy_m\right\|_{Z(Q)}\\
      &\qquad\times \frac{|h_l|^{mn-\alpha}}{\|\chi_{Q}\|_Z|Q|^{\alpha/n}}\rightarrow 0,\,\text{as}\, l \rightarrow \infty\, \text{uniformly for all cubes}\, Q, \,\text{where}\, c_Q\, \text{denotes the center of}\  Q,
    \end{split}
  \end{equation*}

  \end{enumerate}
  then $b\in BMO_\mu$, and $\|b\|_{BMO_\mu}\lesssim \|[b,T_{K_{\alpha}}]_i\|_{X_1\times \cdots\times X_m\rightarrow Y}$.
\end{theorem}

\begin{proof} Without loss of generality, we only deal with the case $0<c<C$.
As in the proof of Theorem \ref{theorem, structure, homogeneous kernel},
a limiting argument is needed since we don't know the local integrability of $b$ in $Y$.
In order to avoid cumbersome,  we omit the limiting argument and assume that $b$ is local integrable in $Y$.

For simplicity in notation and proof, we only present the proof for $m=2$, $i=1$, since the rest cases can be verified similarly.
Here we use $Q_0=[-1/2,1/2]^n$ to denote the unit cube in $\mathbb{R}^{n}$,
and use $Q_0^2=[-1/2,1/2]^{2n}$ to denote the unit cube in $\mathbb{R}^{2n}$.
Since the exact value of $\alpha$ does not affect the proof, for simplicity we abbreviate $T_{K_{\alpha}}$ to $T$, and abbreviate $K_{\alpha}$ to $K$.
Choose a constant $\tau_0$ and a nonempty open cone
$\Gamma\subset \widetilde{\Gamma}\subset \mathbb{R}^{2n}$ with vertex at the origin,
such that for any $u\in 2Q_0^2$, $v\in \Gamma_{\tau_0}:=\Gamma\cap B^c(0,\tau_0)$, we have $u+v\in \widetilde{\Gamma}$.

Furthermore, we can find a sequence $\{h_l=(h_l^1,h_l^2)\}_{l=1}^{\infty}$ satisfying that $h_l\in \Gamma_{\tau_0}$, and $|h_l|\rightarrow \infty$
as $l\rightarrow \infty$, such that
\begin{equation*}
  \begin{split}
  \frac{|h_l|^{2n-\alpha}}{|Q|^{\alpha/n}\|\chi_{Q}\|_Z}\left\|\int_{\prod_{j=1}^2(Q-\sqrt[n]{|Q|}h_l^j)}
  |K(\cdot,y,z)-K(\cdot,c_{Q}-\sqrt[n]{|Q|}h_l^1,z)|dydz\right\|_{Z(Q)}\rightarrow 0
  \end{split}
\end{equation*}
as $l\rightarrow \infty$, uniformly for all cubes $Q$.

For a fixed cube $Q_1$ with side length $\rho$, without loss of generality, we  may assume that
\begin{equation*}
  \int_{Q_1}b(y)dy=0,
\end{equation*} and set
\begin{equation*}
  B:=\frac{1}{\mu(Q_1)}\int_{Q_1}|b(y)|dy.
\end{equation*}
Take
\begin{equation*}
  \phi_1(x)=\left(sgn(b(x))-\frac{1}{|Q_1|}\int_{Q_1}sgn(b(y))dy\right)\chi_{Q_1}(x).
\end{equation*}
Then $-2\chi_{Q_1}\leq \phi_1\leq 2\chi_{Q_1}$, $b\phi_1\geqslant 0$.

Take $Q_{(l)}:=Q_1+\rho h_l^1$, $Q_{2,l}:=Q_{(l)}-\rho h_l^2$, $\phi_2:=\chi_{Q_{2,l}}$.
Then, for $x\in Q_{(l)}$, $y\in Q_1$, $z\in Q_{2,l}$,
we have $$\frac{(x-y,x-z)}{\rho}\in 2Q_0^2+h_l\subset 2Q_0^2+\Gamma_{\tau_0}\subset \widetilde{\Gamma}.$$
Thus,
\begin{equation}\label{for proof, structures, 10}
  K(x,y,z)\sim \frac{1}{(|x-y|+|x-z|)^{2n-\alpha}}\sim \frac{1}{(\rho|h_l|)^{2n-\alpha}}.
\end{equation}
By the assumption $Y\cdot Z\subset \widetilde{Y}$,
  we have
\begin{equation*}
  \begin{split}
    \|[b,T]_1(\phi_1,\phi_2)\|_Y
    \geqslant
    \|[b,T]_1(\phi_1,\phi_2)\chi_{Q_{(l)}}\|_Y
    \gtrsim
    \|[b,T]_1(\phi_1,\phi_2)\chi_{Q_{(l)}}\|_{\widetilde{Y}}/\|\chi_{Q_{(l)}}\|_Z.
  \end{split}
\end{equation*}
Combining this with
\begin{equation*}
  \|[b,T]_1(\phi_1,\phi_2)\chi_{Q_{(l)}}\|_{\widetilde{Y}}\geq A_1\|T(b\phi_1,\phi_2)\chi_{Q_{(l)}}\|_{\widetilde{Y}}-A_2\|bT(\phi_1,\phi_2)\chi_{Q_{(l)}}\|_{\widetilde{Y}},
\end{equation*}
we get
\begin{equation}\label{for proof, structures, 1}
  \|[b,T]_1(\phi_1,\phi_2)\|_Y
  \geqslant
  \frac{A_1\|T(b\phi_1,\phi_2)\chi_{Q_{(l)}}\|_{\widetilde{Y}}-A_2\|bT(\phi_1,\phi_2)\chi_{Q_{(l)}}\|_{\widetilde{Y}}}{\|\chi_{Q_{(l)}}\|_Z}.
\end{equation}
Recalling $b\phi_1\geqslant 0$, we obtain that, for $x\in Q_{(l)}$
\begin{equation*}
  \begin{split}
    |T(b\phi_1,\phi_2)(x)|
    = &
    \int_{Q_1\times Q_{2,l}} K(x,y,z)b(y)\phi_1(y)dydz
    \\
    \gtrsim &
    \frac{1}{(\rho|h_l|)^{2n-\alpha}}\int_{Q_1\times Q_{2,l}} b(y)\phi_1(y)dydz
    \\
    = &
    \frac{1}{(\rho|h_l|)^{2n-\alpha}}\int_{Q_1\times Q_{2,l}} |b(y)|dydz
    =
    \frac{\rho^{\alpha-n}\mu(Q_1)B}{|h_l|^{2n-\alpha}}.
  \end{split}
\end{equation*}
This implies that
\begin{equation}\label{for proof, structures, 2}
  \begin{split}
   \frac{ \|T(b\phi_1,\phi_2)\chi_{Q_{(l)}}\|_{\widetilde{Y}}}{\|\chi_{Q_{(l)}}\|_Z}
    \gtrsim
    \frac{\rho^{\alpha-n}\mu(Q_1)B}{|h_l|^{2n-\alpha}}\cdot \frac{\|\chi_{Q_{(l)}}\|_{\widetilde{Y}}}{\|\chi_{Q_{(l)}}\|_Z}
    \sim
    \frac{\rho^{\alpha-n}\mu(Q_1)B\|\chi_{Q_{(l)}}\|_Y}{|h_l|^{2n-\alpha}}.
  \end{split}
\end{equation}
On the other hand, the assumption $Y\cdot Z\subset \widetilde{Y}$ implies that
\begin{equation}\label{for proof, structures, 3}
  \begin{split}
   \frac{ \|bT(\phi_1,\phi_2)\chi_{Q_{(l)}}\|_{\widetilde{Y}}}{\|\chi_{Q_{(l)}}\|_Z}
    \lesssim
   \frac{\|b\chi_{Q_{(l)}}\|_{Y}\|T(\phi_1,\phi_2)\chi_{Q_{(l)}}\|_{Z}}{\|\chi_{Q_{(l)}}\|_Z}.
  \end{split}
\end{equation}
The combination of (\ref{for proof, structures, 1}), (\ref{for proof, structures, 2}) and (\ref{for proof, structures, 3})
then yields that
\begin{equation}\label{for proof, structures, 8}
  \|[b,T]_1(\phi_1,\phi_2)\|_Y
  \geqslant
  \frac{A_1\rho^{\alpha-n}\mu(Q_1)B\|\chi_{Q_{(l)}}\|_Y}{|h_l|^{2n-\alpha}}-\frac{A_2\|b\chi_{Q_{(l)}}\|_Y\|T(\phi_1,\phi_2)\chi_{Q_{(l)}}\|_Z}{\|\chi_{Q_{(l)}}\|_Z}.
\end{equation}
Take $$\psi_1=\chi_{Q_1},\ \psi_2=\phi_2=\chi_{Q_{2,l}}.$$
We have
\begin{equation}\label{for proof, structures, 5}
  \begin{split}
    \|[b,T]_1(\psi_1,\psi_2)\|_{Y}
    \geqslant &
    \|[b,T]_1(\psi_1,\psi_2)\chi_{Q_{(l)}}\|_{Y}
    \\
    \geqslant &
    A_3\|bT(\psi_1,\psi_2)\chi_{Q_{(l)}}\|_{Y}-A_4\|T(b\psi_1,\psi_2)\chi_{Q_{(l)}}\|_{Y}.
  \end{split}
\end{equation}
Recalling (\ref{for proof, structures, 10}),
\begin{equation*}
  \begin{split}
    |T(b\psi_1,\psi_2)(x)|
    \leqslant &
    \int_{Q_1\times Q_{2,l}} |K(x,y,z)||b(y)|dydz
    \\
    \lesssim &
    \frac{1}{(\rho|h_l|)^{2n-\alpha}}\int_{Q_1\times Q_{2,l}} |b(y)|dydz
    =
    \frac{\rho^{\alpha-n}\mu(Q_1)B}{|h_l|^{2n-\alpha}},\quad\forall\, x\in Q_{(l)}.
  \end{split}
\end{equation*}
This implies that
\begin{equation}\label{for proof, structures, 6}
  \begin{split}
    \|T(b\psi_1,\psi_2)\chi_{Q_{(l)}}\|_{Y}
    \lesssim
    \frac{\rho^{\alpha-n}\mu(Q_1)B\|\chi_{Q_{(l)}}\|_Y}{|h_l|^{2n-\alpha}}.
  \end{split}
\end{equation}
Also, for $x\in Q_{(l)}$,
\begin{equation*}
  \begin{split}
    |b(x)T(\psi_1,\psi_2)(x)|
    =&
    \left|b(x)\int_{Q_1\times Q_{2,l}}K(x,y,z)dydz\right|
    \\
    \gtrsim &
    \frac{|b(x)|\cdot |Q_1|\cdot |Q_{2,l}|}{(\rho|h_l|)^{2n-\alpha}}=\frac{\rho^{\alpha}|b(x)|}{|h_l|^{2n-\alpha}}.
  \end{split}
\end{equation*}
Consequently,
\begin{equation}\label{for proof, structures, 7}
  \begin{split}
    \|bT(\psi_1,\psi_2)\chi_{Q_{(l)}}\|_{Y}
    \gtrsim
    \frac{\rho^{\alpha}\|b\chi_{Q_{(l)}}\|_Y}{|h_l|^{2n-\alpha}}.
  \end{split}
\end{equation}
The combination of (\ref{for proof, structures, 5}), (\ref{for proof, structures, 6}) and (\ref{for proof, structures, 7}) yields that
\begin{equation}\label{for proof, structures, 9}
  \|[b,T]_1(\psi_1,\psi_2)\|_Y
  \geqslant
  \frac{A_3\rho^{\alpha}\|b\chi_{Q_{(l)}}\|_Y}{|h_l|^{2n-\alpha}}-\frac{A_4\rho^{\alpha-n}\mu(Q_1)B\|\chi_{Q_{(l)}}\|_Y}{|h_l|^{2n-\alpha}}.
\end{equation}
Denote
\begin{equation*}
  \Xi =\frac{A_2|h_l|^{2n-\alpha}\|T(\phi_1,\phi_2)\chi_{Q_{(l)}}\|_Z}{A_3\rho^{\alpha}\|\chi_{Q_{(l)}}\|_Z}.
\end{equation*}
Using (\ref{for proof, structures, 8}),(\ref{for proof, structures, 9}), and the boundedness of $[b,T]$, we obtain that
\begin{equation*}
  \begin{split}
    (1+\Xi)\|[b,T]_1\|_{X_1\times X_2\rightarrow Y}&(\|\phi_1\|_{X_1}\|\phi_2\|_{X_2}+\|\psi_1\|_{X_1}\|\psi_2\|_{X_2})
    \\
    \geqslant &
    \|[b,T]_1(\phi_1,\phi_2)\|_Y+\Xi\|[b,T]_1(\psi_1,\psi_2)\|_Y
    \\
    \geqslant &
    (A_1-\Xi A_4)\frac{\rho^{\alpha-n}\mu(Q_1)B\|\chi_{Q_{(l)}}\|_Y}{|h_l|^{2n-\alpha}}.
  \end{split}
\end{equation*}
If we can make $\Xi\leqslant \min\{A_1/(2A_4),1\}$, we get $A_1-\Xi A_4\geqslant \frac{A_2}{2}$,
then
\begin{equation*}
  \begin{split}
    \frac{A_1}{2}\,\cdot\,\frac{\rho^{\alpha-n}\mu(Q_1)B\|\chi_{Q_{(l)}}\|_Y}{|h_l|^{2n-\alpha}}
    \leqslant &
    (1+\Xi)\|[b,T]_1\|_{X_1\times X_2\rightarrow Y}(\|\phi_1\|_{X_1}\|\phi_2\|_{X_2}+\|\psi_1\|_{X_1}\|\psi_2\|_{X_2})
    \\
    \leqslant &
    3(1+\Xi)\|[b,T]_1\|_{X_1\times X_2\rightarrow Y}\|\chi_{Q_1}\|_{X_1}\|\chi_{Q_{2,l}}\|_{X_2}.
  \end{split}
\end{equation*}
In what follows, we will prove the desired result only under the condition $(b)$, since the arguments under the condition $(a)$ is similar.
Observing that $Q_1\subset \lambda Q_{(l)}=2(|y_0-x_0|+\sqrt{n})Q_{(l)}=2(|h_l^1|+\sqrt{n})Q_{(l)}$,
we have
\begin{equation*}
  \|\chi_{Q_1}\|_{X_1}\leqslant \|\chi_{\lambda Q_{(l)}}\|_{X_1}\leqslant C_{|h_l^1|}\|\chi_{Q_{(l)}}\|_{X_1}.
\end{equation*}
Using similar technique, we can also deduce
\begin{equation*}
  \|\chi_{Q_{2,l}}\|_{X_2} \lesssim_{l}\|\chi_{Q_{(l)}}\|_{X_2},\ \mu(Q_{(l)})\lesssim_l \mu(Q_1).
\end{equation*}
By the assumption that $\|\chi_{Q_{(l)}}\|_{X_1}\|\chi_{Q_{(l)}}\|_{X_2}\lesssim \|\chi_{Q_{(l)}}\|_Y|Q_{(l)}|^{\alpha/n-1}\mu(Q_{(l)})$,
we have
\begin{equation*}
  \begin{split}
    \|\chi_{Q_1}\|_{X_1}\|\chi_{Q_{2,l}}\|_{X_2}
    \lesssim_l &
    \|\chi_{Q_{(l)}}\|_{X_1}\|\chi_{Q_{(l)}}\|_{X_2}
    \\
    \lesssim_l &
    \|\chi_{Q_{(l)}}\|_Y|Q_{(l)}|^{\alpha/n-1}\mu(Q_{(l)})
        \\
    \lesssim_l &
    \|\chi_{Q_{(l)}}\|_Y|Q_{(l)}|^{\alpha/n-1}\mu(Q_{1}).
  \end{split}
\end{equation*}
Thus,
\begin{equation*}
  \begin{split}
    \frac{A_1}{2}\cdot\,\frac{\rho^{\alpha-n}\mu(Q_1)B\|\chi_{Q_{(l)}}\|_Y}{|h_l|^{2n-\alpha}}
    \leqslant&
    3(1+\Xi)\|[b,T]_1\|_{X_1\times X_2\rightarrow Y}\|\chi_{Q_1}\|_{X_1}\|\chi_{Q_{2,l}}\|_{X_2}
    \\
    \lesssim_l &
    \|[b,T]_1\|_{X_1\times X_2\rightarrow Y}\|\chi_{Q_{(l)}}\|_Y|Q_{(l)}|^{\alpha/n-1}\mu(Q_{1}),
  \end{split}
\end{equation*}
which implies that
\begin{equation*}
  B\lesssim_l |h_l|^{2n-\alpha}\cdot \|[b,T]_1\|_{X_1\times X_2\rightarrow Y}.
\end{equation*}

The remaining thing is to check that $\Xi$ can be chosen small for sufficient large $l$.
For $x\in Q_{(l)}$, we have
\begin{equation*}
  \begin{split}
    |T(\phi_1,\phi_2)|
    = &
    \Big|\int_{Q_1\times Q_{2,l}}K(x,y,z)\phi_1(y)dydz\Big|
    \\
    \leqslant &
    2\int_{Q_1\times Q_{2,l}}\left|K(x,y,z)-K(x,c_{Q_1},z)\right|dydz.
  \end{split}
\end{equation*}
Then,
\begin{equation*}
  \begin{split}
    \|T(\phi_1,\phi_2)\chi_{Q_{(l)}}\|_Z
    \lesssim &
    \left\|\int_{Q_1\times Q_{2,l}}|K(\cdot,y,z)-K(\cdot,c_{Q_1},z)|dydz\right\|_{Z(Q_{(l)})}
  \end{split}
\end{equation*}
This shows that
\begin{equation*}
  \begin{split}
    \Xi
    \lesssim &
    \frac{|h_l|^{2n-\alpha}}{\rho^{\alpha}\|\chi_{Q_{(l)}}\|_Z}\left\|\int_{Q_1\times Q_{2,l}}|K(\cdot,y,z)-K(\cdot,c_{Q_1},z)|dydz\right\|_{Z(Q_{(l)})}\rightarrow 0\ \text{as}\  l\rightarrow \infty.
  \end{split}
\end{equation*}
We have now completed the proof of Theorem 3.1.
\end{proof}

As in the linear setting, the pointwise estimate is also useful for the endpoint cases in multilinear setting.

\begin{proposition}[Technique, standard kernel, pointwise estimates]\label{proposition, technique, standard kernel, pointwise}
Let $m\in\mathbb{N}$. For given $i\in\{1,2,\cdots,m\}$, let $K_{\alpha}$ be a function defined off the diagonal $x=y_1=\cdots =y_m$ in $(\mathbb{R}^n)^{m+1}$, satisfying the following local property in an open cone $\widetilde{\Gamma}$ of $(\mathbb{R}^{n})^m$ with vertex at
the origin:
  \begin{enumerate}[(i)]
  \item lower and upper bound: \\$\frac{c}{(\sum_{j=1}^m|x-y_j|)^{mn-\alpha}}\leqslant K_{\alpha}(x,y_1,\cdots,y_m) \leqslant \frac{C}{(\sum_{j=1}^m|x-y_j|)^{mn-\alpha}}$
  for all $(x-y_1,\cdots,x-y_m)\in \widetilde{\Gamma}$, where $0<c<C$ or $c<C<0$,
  \item
  for any open cone $\overline{\Gamma}\subset \widetilde{\Gamma}$,
  there exists a sequence $\{h_l=(h_l^1,\cdots,h_l^m)\}_{j=1}^{\infty}$
  satisfying that $h_l\in \overline{\Gamma}$, $|h_l|\rightarrow \infty$ as $l\rightarrow \infty$, and
  \begin{equation*}
    \begin{split}
      &\left\|\int_{\prod\limits_{{j=1}}^m(Q-\sqrt[n]{Q}h_l^j)}
      |K_{\alpha}(\cdot,y_1,\cdots,y_m)-K_{\alpha}(\cdot,y_1,\cdots,y_{i-1},c_{Q}-\sqrt[n]{Q}h_l^i,y_{i+1},\cdots,y_m)|dy_1\cdots\,dy_m\right\|_{L^{\infty}(Q)}
      \\
 &\qquad\times \frac{|h_l|^{mn-\alpha}}{|Q|^{\alpha/n}}\rightarrow 0,
  \ \text{as}\ l \rightarrow \infty\ \text{uniformly for all cubes}\ Q,\, {\rm where}\, c_Q \,{\rm denotes\, the\, center\, of\,} Q.
    \end{split}
  \end{equation*}
  \end{enumerate}
Let $T_{K_{\alpha}}$ be a m-linear operator associated with $K_\alpha$, $b\in L_{loc}^1(\mathbb{R}^n)$, $[b,T_{K_{\alpha}}]_i$ be the $i-th$ commutator formed by  $T_{K_{\alpha}}$ with $b$. Then, for any cube denoted by $Q_i\subset \mathbb{R}^n$, there exist two function sequences $\{\phi_j\}_{j=1}^m$ and $\{\psi_j\}_{j=1}^m$ satisfying
  $|\phi_j|,|\psi_j|\leq 2\chi_{Q_j} (j=1,\cdots,m)$, where $Q_j (j\neq i)$ are certain cubes with the same side length of $Q_i$,
  and there exists a cube $Q$ with the same side length of $Q_i$,
  $Q_j\subset \lambda Q$ for all $j=1,2,\cdots,m$ and some $\lambda>0$ independent of $Q_i$,
  such that
\begin{equation*}
  \begin{split}
   \frac{1}{|Q_i|^{1-\alpha/n}}\int_{Q_i}|b(y)-b_{Q_i}|dy
    \leq \widetilde{C}
    \left(|[b,T_{K_\alpha}]_i(\phi_1,\cdots,\phi_m)(x)|+|[b,T_{K_\alpha}]_i(\psi_1,\cdots,\psi_m)(x)|\right),\quad\forall\, x\in Q,
  \end{split}
\end{equation*}
where the constant $\widetilde{C}$ is independent of the choice of $Q_i$.
\end{proposition}

\begin{proof}
For simplicity in notation and proof, we only give the arguments for $m=2$, $i=1$, since the cases for $m\geq 3$ can be treated similarly.
  Without loss of generality, we may assume that
  \begin{equation*}
    \frac{c}{(|x-y|+|x-z|)^{2n-\alpha}}\leqslant K_{\alpha}(x,y,z) \leqslant \frac{C}{(|x-y|+|x-z|)^{2n-\alpha}}
  \end{equation*}
    for all $(x-y,x-z)\in \widetilde{\Gamma}$, where $0<c<C$.
 We use $Q_0=[-1/2,1/2]^n$ to denote the unit cube in $\mathbb{R}^{n}$,
and $Q_0^2=[-1/2,1/2]^{2n}$ the unit cube in $\mathbb{R}^{2n}$.
Since the exact value of $\alpha$ does not affect the proof,
we abbreviate $T_{K_{\alpha}}$ to $T$, and abbreviate $K_{\alpha}$ to $K$.
Choose a constant $\tau_0$ and a nonempty open cone
$\Gamma\subset \widetilde{\Gamma}\subset \mathbb{R}^{2n}$ with vertex at the origin,
such that for any $u\in 2Q_0^2$, $v\in \Gamma_{\tau_0}=\Gamma\cap B^c(0,\tau_0)$, we have $u+v\in \widetilde{\Gamma}$.

Furthermore, we can find a sequence $\{h_l=(h_l^1,h_l^2)\}_{l=1}^{\infty}$ satisfying that $h_l\in \Gamma_{\tau_0}$, and $|h_l|\rightarrow \infty$
as $l\rightarrow \infty$, such that
\begin{equation}\label{for proof, pointwise, 4}
  \begin{split}
  \frac{|h_l|^{2n-\alpha}}{|Q|^{\alpha/n}}\left\|\int_{\prod_{j=1}^2(Q-\sqrt[n]{|Q|}h_l^j)}
  |K(\cdot,y,z)-K(\cdot,c_{Q}-\sqrt[n]{|Q|}h_l^1,z)|dydz\right\|_{L^{\infty}(Q)}\rightarrow 0
  \end{split}
\end{equation}
as $l\rightarrow \infty$, uniformly for all cubes $Q$.

Let $Q_1$ be a fixed cube with side length $\rho$. Without loss of generality, we may assume
\begin{equation*}
  \int_{Q_1}b(y)dy=0,
\end{equation*}
and set
\begin{equation}
  B:=\frac{1}{|Q_1|}\int_{Q_1}|b(y)|dy.
\end{equation}
Take
\begin{equation*}
  \phi_1(x)=\Big(sgn(b(x))-\frac{1}{|Q_1|}\int_{Q_1}sgn(b(y))dy\Big)\chi_{Q_1}(x).
\end{equation*}
Then $|\phi_1|\leq 2\chi_{Q_1}$, $b\phi\geq 0$.
Let $Q_{(l)}:=Q_1+\rho h_l^1$, $Q_{2,l}:=Q_{(l)}-\rho h_l^2$, and $\phi_2:=\chi_{Q_{2,l}}$. As in the proof of Theorem \ref{Theorem, structure, standard kernel}, we can deduce that
\begin{equation*}
  \begin{split}
    |T(b\phi_1,\phi_2)(x)|
    \gtrsim &
    \frac{1}{(\rho|h_l|)^{2n-\alpha}}\int_{Q_1\times Q_{2,l}} b(y)\phi_1(y)dydz
    \\
    = &
    \frac{1}{(\rho|h_l|)^{2n-\alpha}}\int_{Q_1\times Q_{2,l}} |b(y)|dydz
    =
    \frac{\rho^{\alpha-n}|Q_1|B}{|h_l|^{2n-\alpha}},\quad\forall\,x\in Q_{(l)}.
  \end{split}
\end{equation*}
Hence,
\begin{equation*}
  \begin{split}
    |[b,T]_1(\phi_1,\phi_2)(x)|
    \geq &
    |T(b\phi_1,\phi_2)(x)|-|b(x)T(\phi_1,\phi_2)(x)|
    \\
    \geq &
    \frac{A_1\rho^{\alpha-n}|Q_1|B}{|h_l|^{2n-\alpha}}-|b(x)T(\phi_1,\phi_2)(x)|,\quad\forall\,x\in Q_{(l)}.
  \end{split}
\end{equation*}
Moreover, take $\psi_1:=\chi_{Q_1},\ \psi_2:=\phi_2=\chi_{Q_{2,l}}$.
As in the proof of Theorem \ref{Theorem, structure, standard kernel}, we deduce that, for $x\in Q_{(l)}$,
\begin{equation*}
  \begin{split}
    |T(b\psi_1,\psi_2)(x)|
    \leqslant &
    |\int_{Q_1\times Q_{2,l}} K(x,y,z)|b(y)|dydz
    \\
    \lesssim &
    \frac{1}{(\rho|h_l|)^{2n-\alpha}}\int_{Q_1\times Q_{2,l}} |b(y)|dydz
    =
    \frac{\rho^{\alpha-n}|Q_1|B}{|h_l|^{2n-\alpha}},
  \end{split}
\end{equation*}
and
\begin{equation*}
  \begin{split}
    |b(x)T(\psi_1,\psi_2)(x)|
    =&
    \left|b(x)\int_{Q_1\times Q_{2,l}}K(x,y,z)dydz\right|
    \\
    \gtrsim &
    \frac{|b(x)|\cdot |Q_1|\cdot |Q_{2,l}|}{(\rho|h_l|)^{2n-\alpha}}=\frac{\rho^{\alpha}|b(x)|}{|h_l|^{2n-\alpha}}.
  \end{split}
\end{equation*}
This implies that
\begin{equation*}
  \begin{split}
    |[b,T]_1(\psi_1,\psi_2)(x)|
    \geq &
    |b(x)T(\psi_1,\psi_2)(x)|-|T(b\psi_1,\psi_2)(x)|
    \\
    \geq &
    \frac{A_3\rho^{\alpha}|b(x)|}{|h_l|^{2n-\alpha}}-\frac{A_4\rho^{\alpha-n}|Q_1|B}{|h_l|^{2n-\alpha}},\quad\forall\,x\in Q_{(l)}.
  \end{split}
\end{equation*}
Denote
\begin{equation}
  \Xi(x): =\frac{A_2|h_l|^{2n-\alpha}|T(\phi_1,\phi_2)(x)\chi_{Q}(x)|}{A_3\rho^{\alpha}}.
\end{equation}
By (\ref{for proof, pointwise, 4}) and the same arguments as in the proof of Theorem \ref{Theorem, structure, standard kernel},
we take sufficient large $l=\mathbf{l_0}$ independent of $Q_1$, to ensure
\begin{equation*}
  \Xi(x)\leq \min\{\frac{A_1}{2A_4},1\},\ \forall x\in Q_{(l)}.
\end{equation*}
Then
\begin{equation*}
  \begin{split}
    |[b,T]_1(\phi_1,\phi_2)(x)|+|b(x)T(\psi_1,\psi_2)(x)|
    \geq &
    |[b,T]_1(\phi_1,\phi_2)(x)|+\Xi(x)|b(x)T(\psi_1,\psi_2)(x)|
    \\
    \geq &
    \frac{A_1\rho^{\alpha-n}|Q_1|B}{2|h_l|^{2n-\alpha}}
    =
    \frac{A_1}{2|h_l|^{2n\alpha}|Q_1|^{1-\alpha/n}}\int_{Q_1}|b(y)|dy
  \end{split}
\end{equation*}
for all $x\in Q_{(\mathbf{l_0})}=Q_1+\rho h_{\mathbf{l_0}}^1$.
Take $Q=Q_{(\mathbf{l_0})}$, $Q_2=Q_{2,\mathbf{l_0}}$. Recall the side length of $Q_1$ is $\rho$.
So, there exists $\lambda$ depend only on $h_{\mathbf{l_0}}$,
such that $Q_1\subset \lambda Q$ and $Q_2\subset \widetilde{\lambda} Q$.
We have now completed this proof.
\end{proof}

\begin{proposition}[Technique, standard kernel, multilinear, $Z=L^{\infty}$]\label{proposition, technique, standard kernel, Z=L^infty}
Let $K_{\alpha}$ be a function defined off the diagonal $x=y_1=\cdots =y_m$ in $(\mathbb{R}^n)^{m+1}$, satisfying
\begin{equation*}
  |K_{\alpha}(x,y_1,\cdots,y_i,\cdots,y_m)-K_{\alpha}(x,y_1,\cdots,y_i',\cdots,y_m)|
  \lesssim \frac{|y_i-y_i'|^{\delta}}{(\sum_{j=1}^m|x-y_j|)^{mn-\alpha+\delta}},
\end{equation*}
whenever $|y_i-y_i'|\leqslant \frac{1}{2}\max_{1\leqslant j\leqslant m}|x-y_j|$.
Then, for any open cone $\widetilde{\Gamma}\subset (\mathbb{R}^n)^m$, there exists an open cone $\Gamma\subset \widetilde{\Gamma}$
such that for every $\{h_l=(h_l^1,\cdots,h_l^m)\}_l\subset \Gamma$ satisfying $|h_l|\rightarrow \infty$ as $l\rightarrow \infty$, we have
\begin{equation*}
    \begin{split}
      &\left\|\int_{\prod\limits_{{j=1}}^m(Q-\sqrt[n]{Q}h_l^j)}
      |K_{\alpha}(\cdot,y_1,\cdots,y_m)-K_{\alpha}(\cdot,y_1,\cdots,y_{i-1},c_{Q}-\sqrt[n]{|Q|}h_l^i,y_{i+1},\cdots,y_m)|dy_1\cdots\,dy_m\right\|_{L^{\infty}(Q)}
      \\
  &\qquad\times\frac{|h_l|^{mn-\alpha}}{|Q|^{\alpha/n}}\rightarrow 0,
  \ \text{as}\ l \rightarrow \infty\ \text{uniformly for all cubes}\ Q\subset\mathbb{R}^n,\, \text{where}\, c_Q\, \text{denotes the center of}\, Q.
    \end{split}
  \end{equation*}
\end{proposition}

\begin{proof}
Firstly, we verify that
\begin{equation}\label{for proof, technieuq, 1}
  \begin{split}
    &\left\|\int_{\prod\limits_{{j=1}}^mQ_j}|K_{\alpha}(\cdot,y_1,\cdots,y_m)-K_{\alpha}(\cdot,y_1,\cdots,y_{i-1},c_{Q_i},y_{i+1},\cdots,y_m)|dy_1\cdots,dy_m\right\|_{L^{\infty}(Q)}
    \\
    &\qquad\times    \frac{|c_{Q}-c_{Q_i}|^{mn-\alpha}}{|Q|^m}\rightarrow 0
  \end{split}
\end{equation}
uniformly for all cubes $Q$, $Q_j\,(j=1,\cdots,m)$ with $|Q|=|Q_j|,\, j=1,\cdots, m$,
  as $|c_Q-c_{Q_i}|/\sqrt[n]{|Q|}\rightarrow \infty$.

When the quantity $|c_Q-c_{Q_i}|/\sqrt[n]{|Q|} $ sufficient large,
if $x\in Q$, $y_i\in Q_i$, and $|Q|=|Q_1|$, we have
\begin{equation}
  |y_i-c_{Q_i}|\leqslant \frac{1}{2}|x-y_i|\leqslant \frac{1}{2}\max_{1\leqslant j\leqslant m}|x-y_j|,
\end{equation}
which implies that
\begin{equation*}
  \begin{split}
    |K_{\alpha}(x,y_1,\cdots,y_m)-K_{\alpha}(x,y_1,\cdots,y_{i-1},c_{Q_i},y_{i+1},\cdots,y_m)|
    \lesssim &
    \frac{|y_i-c_{Q_i}|^{\delta}}{(\sum_{j=1}^m|x-y_j|)^{mn-\alpha+\delta}}
    \\
    \lesssim &
    \frac{|Q_i|^{\delta/n}}{|c_Q-c_{Q_i}|^{mn-\alpha+\delta}}.
  \end{split}
\end{equation*}
Thus, for $x\in Q$,
\begin{equation*}
  \begin{split}
    &\frac{|c_{Q}-c_{Q_i}|^{mn-\alpha}}{|Q|^m}\int_{\prod\limits_{{j=1}}^mQ_i}|K_{\alpha}(x,y_1,\cdots,y_m)
    -K_{\alpha}(x,\cdots,y_{i-1},c_{Q_i},y_{i+1},\cdots,y_m)|dy_1\cdots\,dy_m\\
    &\qquad\lesssim
    \frac{|c_{Q}-c_{Q_i}|^{mn-\alpha}}{|Q|^m}\cdot\frac{|Q_i|^{\delta/n}}{|c_Q-c_{Q_i}|^{mn-\alpha+\delta}}\cdot |Q|^m
   =
    \frac{|Q_i|^{\delta/n}}{|c_Q-c_{Q_i}|^{\delta}}\rightarrow 0,\quad\text{as}\,\,\frac {|c_Q-c_{Q_i}|}{\sqrt[n]{|Q|}}\to \infty.
  \end{split}
\end{equation*}

For every open cone $\widetilde{\Gamma}$, choose an open cone $\Gamma\subset \widetilde{\Gamma}$ such that for every
$h=(h^1,\cdots,h^j,\cdots,h^m)\in \Gamma$, we have $|h|\sim |h^i|$, where $h^i\in \mathbb{R}^n$.
Let $Q_j=Q-\sqrt[n]{|Q|}h_l^j$ for every $j=1,\cdots,m$. Using (\ref{for proof, technieuq, 1}),
$|c_Q-c_{Q_i}|/\sqrt[n]{|Q|}=|h_l^i|$
and the fact $|h_l|\sim |h_l^i|$, we deduce that
\begin{equation*}
    \begin{split}
      &\left\|\int_{\prod\limits_{{j=1}}^m(Q-\sqrt[n]{Q}h_l^j)}
      |K_{\alpha}(\cdot,y_1,\cdots,y_m)-K_{\alpha}(\cdot,y_1,\cdots,y_{i-1},c_{Q}-\sqrt[n]{|Q|}h_l^i,y_{i+1},\cdots,y_m)|dy_1\cdots\,dy_m\right\|_{L^{\infty}(Q)}
      \\
  &\qquad\qquad\times \frac{|h_l|^{mn-\alpha}}{|Q|^{\alpha/n}}\rightarrow 0,
  \ \text{as}\ l \rightarrow \infty\ \text{uniformly for all cubes}\ Q\subset \mathbb{R}^n.
    \end{split}
\end{equation*}
This completes the proof of Proposition 3.3.
\end{proof}

\begin{remark}
  As in the linear setting, in most of the cases (unweighted, non endpoint cases), $Z=L^1$ has its advantage over $Z=L^{\infty}$,
  since the assumption of kernel with $Z=L^{1}$ is weaker than that with $Z=L^{\infty}$.
  However, in most of the previous works for multilinear commutators, the corresponding kernel is assumed to be "smooth enough",
  which can be handled by using $Z=L^{\infty}$. Therefore, we here deal with the situation only for $Z=L^{\infty}$, and
 the corresponding technique proposition for the case of  $Z=L^1$ in multilinear setting can be established similarly like in linear setting.
\end{remark}

\section{Examples and applications}
By the basic assumption of Quasi-Banach space, $Y\cdot L^{\infty}\subset Y$ is naturally established.
Thus, we can take $Z=L^{\infty}$, $\widetilde{Y}=Y$ in Theorems \ref{theorem, structure, homogeneous kernel} and \ref{Theorem, structure, standard kernel}.
However, the corresponding conditions of kernel can be further weakened if the auxiliary space is chosen to be $Z=L^1$.
In this section, we will show that for a large number of important space $Y$, not only for $Z=L^{\infty}$ but also for $Z=L^1$
there exists a Quasi-Banach space $\widetilde{Y}$ such that $\widetilde{Y}\cdot Z\subset Y$.
Moreover, these pairs of Quasi-Banach spaces satisfy the assumptions
in Theorems \ref{theorem, structure, homogeneous kernel} and \ref{Theorem, structure, standard kernel}.

We remark that in the weighted cases, our trick for $Z=L^1(\mathbb{R}^n)$ can not be used, since it is quite hard to find suitable $\widetilde{Y}$ such that
$\|\chi_{Q}\|_{\widetilde{Y}}\sim \|\chi_{Q}\|_{Y}\cdot \|\chi_{Q}\|_{L^1}$ in this case.
In fact, in Theorem \ref{theorem, structure, homogeneous kernel},
if $Y$ is a weighted space, the auxiliary spaces $\widetilde{Y}$ and $Z$ must also be weighted spaces,
unless we choose $Z=L^{\infty}$. So, in the weighted cases below, we will take $Z=L^{\infty}$ to deal with the corresponding results.

In what follows, we will apply our general theory to present some new characterizations of ${\rm BMO}_\mu$ in various settings.

\subsection{$BMO$ and Commutators in Linear Setting}

\subsubsection{$BMO$ and commutators in weak-type Lebesgue spaces}

The investigation on the boundedness and characterization of commutators has been paid lots of attention since the celebrated work established by Coifman, Rochberg and Weiss in \cite{CRW-1976}. In particular, we can found the following known results:

\medskip

{\bf Theorem C}.\quad{\it Let $1<p<\infty$, $T_0$ be the singular integral operator associated to $\Omega$ with the homogenous of degree $0$. Suppose $\Omega\in {\rm Lip}_1(\mathbb{S}^{n-1})$ satisfying (\ref{mean value zero}), $b\in \bigcup_{q> 1}L_{\rm loc}^q(\mathbb{R}^n)$. Then
\begin{enumerate}[(i)]
  \item {\rm (cf. \cite{CRW-1976,Jan})}
$b\in BMO(\mathbb{R}^n)\,\Leftrightarrow \, [b,T_\Omega]: \, L^p(\mathbb{R}^n)\rightarrow\, L^p(\mathbb{R}^n);$
\item {\rm (cf. \cite{Perez-1995})}\quad $b\in BMO(\mathbb{R}^n)\,\Rightarrow\, |\{x\in\mathbb{R}^n:\, |[b,T_\Omega]f(x)|>\lambda\}|\lesssim \displaystyle\int_{\mathbb{R}^n}\Phi(\frac{|f(x)|}{\lambda})dx$ for any $\lambda>0$, where $\Phi(t)=t(1+\log^+t)$;
    \item {\rm(cf. \cite{Chanillo-1982})}\quad $b\in BMO(\mathbb{R}^n)\,\Leftrightarrow \, [b,I_\alpha]: \, L^p(\mathbb{R}^n)\rightarrow\, L^q(\mathbb{R}^n),$
    for $0<\alpha<n$, $1<p<n/\alpha$, $1/q=1/p-\alpha/n$.
\end{enumerate}
}

\medskip

Applying our general theorem, we will weaken the condition of $\Omega$ and establish the characterization of $BMO$ via the boundedness of $[b,T_\alpha]$ in the weak-type Lebesgue spaces, including the endpoint case $p=1$. At first, we present an auxiliary lemma.

\begin{lemma}\label{lemma, auxiliary space, Lorentz}
  Suppose $0<p,q\leq\infty$, $1/\widetilde{p}=1+1/p, 1/\widetilde{q}=1+1/q$.
  Let $Y=L^{p,q}$ and $\widetilde{Y}=L^{\widetilde{p},\widetilde{q}}$ (the Lorentz spaces), $Z=L^1$.
  Then
\begin{equation*}
  Y\cdot Z\subset\widetilde{Y},\hspace{6mm}\|\chi_{Q}\|_{\widetilde{Y}}\sim \|\chi_{Q}\|_{Y}\cdot \|\chi_{Q}\|_{Z}.
\end{equation*}
\end{lemma}
\begin{proof}
Recalling the definition of Lorentz space:
\begin{equation*}
  \|f\|_{L^{p,q}}=\|t^{1/p}f^*(t)\|_{L^q(\mathbb{R}^+,\frac{dt}{t})},
\end{equation*}
where $f^*$ is the decreasing rearrangement of $f$.

Let $1/p=1/p_1+1/p_2$ and $1/q=1/q_1+1/q_2$.
Using the property: $(fg)^*(t_1+t_2)\leqslant f^*(t_1)g^*(t_2)$,
we obtain that
\begin{equation*}
  \begin{split}
    \|fg\|_{\widetilde{Y}}
    =
    \|fg\|_{L^{\widetilde{p},\widetilde{q}}}
    = &
    \|t^{1/\widetilde{p}}(fg)^*(t)\|_{L^{\widetilde{q}}(\mathbb{R}^+,\frac{dt}{t})}
    \\
    \lesssim &
    \|tf^*(t/2)t^{1/p}g^*(t/2)\|_{L^{\widetilde{q}}(\mathbb{R}^+,\frac{dt}{t})}
    \\
    \lesssim &
    \|tf^*(t/2)\|_{L^{1}(\mathbb{R}^+,\frac{dt}{t})}
    \|t^{1/p}g^*(t/2)\|_{L^{q}(\mathbb{R}^+,\frac{dt}{t})}
    \\
    \lesssim &
    \|f\|_{L^{1,1}}\|g\|_{L^{p,q}}=\|f\|_{Z}\|g\|_Y.
  \end{split}
\end{equation*}
Moreover,
\begin{equation*}
  \begin{split}
    \|\chi_Q\|_{\widetilde{Y}}\sim |Q|^{1/\widetilde{p}}
    =
    |Q|\cdot |Q|^{1/p}
    \sim \|\chi_Q\|_{Z}\|\chi_Q\|_{Y}.
  \end{split}
\end{equation*}
This completes the proof of Lemma 4.1.
\end{proof}

Using Proposition \ref{proposition, structure and technique for Z=L^1} and  Lemma \ref{lemma, auxiliary space, Lorentz}, we immediately obtain the following corollary, which is a great improvement and extension to Theorem C.

\begin{corollary}\label{Cor.4.1}
  Let $b\in L_{loc}^1(\mathbb{R}^n)$, $0\le \alpha<n$, $1<p<n/\alpha$, $1/q=1/p-\alpha/n$. 
  Suppose that $\Omega\in L^{\infty}(\mathbb{S}^{n-1})$ is a nonzero homogeneous function of degree 0
  and satisfies (\ref{mean value zero}) for $\alpha=0$. If there exist some open $E\subset\mathbb{S}^{n-1}$, constants $c$  and $C$ such that
  \begin{equation}\label{l4.1}
  c\leqslant \Omega(x')\leqslant C,\quad \forall \,x'\in E,
  \end{equation}
   where $0<c<C$ or $c<C<0$, then the following statements are equivalent:
\begin{enumerate}[(1)]
  \item $b\in BMO(\mathbb{R}^n)$,
  \item $[b,T_{\alpha}]$ is bounded from $L^p(\mathbb{R}^n)$ to $L^q(\mathbb{R}^n)$.
  \item $[b,T_{\alpha}]$ is bounded from $L^p(\mathbb{R}^n)$ to $L^{q,\infty}(\mathbb{R}^n)$.
\end{enumerate}

\end{corollary}

\begin{proof}
  $(1)\Rightarrow (2)\Rightarrow (3)$ directly follows from \cite[Theorem 1]{Hu-2003} and \cite[Theorem 3.6.1]{LDY} with the fact that $L^q(\mathbb{R}^n)\subset L^{q,\infty}(\mathbb{R}^n)$.

 Next, we verify that $(3)\Rightarrow (1)$.  Take $X=L^p(\mathbb{R}^n)$, $Y=L^{p,\infty}(\mathbb{R}^n)$, $\widetilde{Y}=L^{\tilde{p},1}(\mathbb{R}^n)$
  where $1/\widetilde{p}=1/p+1$, $Z=L^1(\mathbb{R}^n)$,
  and take $\mu(Q)=|Q|$, then the desired conclusion follows immediately from Proposition \ref{proposition, structure and technique for Z=L^1} and Lemma \ref{lemma, auxiliary space, Lorentz}.
\end{proof}

\begin{remark}\quad Notice that $\Omega\in C(\mathbb{S}^{n-1})$ implies (\ref{l4.1}) holds. Therefore, replacing (\ref{l4.1}) by that $\Omega\in C(\mathbb{S}^{n-1})$,  Corollary \ref{Cor.4.1} is also true, which can be regarded as an essential improvement of \cite{WZT1}.

\end{remark}

%\begin{corollary}
 % Let $b\in L_{loc}^1(\mathbb{R}^n)$, $1<p<\infty$. Suppose $\Omega\in C(\mathbb{S}^{n-1})$ is a nonzero homogeneous function of degree 0
  %satisfying (\ref{mean value zero}). Then the following statements are equivalent:
%\begin{enumerate}[(1)]
%  \item $b\in BMO(\mathbb{R}^n)$,
%  \item $[b,T_{0}]$ is bounded on $L^p(\mathbb{R}^n)$.
%  \item $[b,T_{0}]$ is bounded from $L^p(\mathbb{R}^n)$ to $L^{p,\infty}(\mathbb{R}^n)$.
%\end{enumerate}
%\end{corollary}

For the endpoint case $p=1$, we can obtain the converse result of (ii) in Theorem C, and get the characterization theorem of $BMO$ via the weak-$L\log^+L$ type boundedness of $[b,T_0]$, which is new characterization of $BMO$ space and the proof will be presented in the next subsection for more general weighted cases.

\begin{corollary}\label{cor.4.2}
Let $\Phi(t)=t(1+\log^+t)$. Suppose $\Omega\in {\rm Lip}(\mathbb{S}^{n-1})$ and $b\in L_{\rm loc}^1(\mathbb{R}^n)$. Then
$$b\in BMO(\mathbb{R}^n)\,\Longleftrightarrow\, |\{x\in \mathbb{R}^n: |[b,T_0](f)(x)|>\lambda\}|\lesssim\displaystyle\int_{\mathbb{R}^n}\Phi\left(\frac{|f(y)|}{\lambda}\right)dy,\,\forall\,\lambda>0.$$
\end{corollary}

\subsubsection{$BMO$ and commutators in one-weight setting}

\medskip

This subsection is concerned with the weighted boundedness of $[b,T_\alpha]$. A lot of attentions has been paid on this topic. We list several relevant results as follows.

\medskip

{\bf Theorem D} (cf. \cite[Theorems 2.4.4 and 3.6.1]{LDY}).\quad{\it Let $0\le \alpha<n$, $1<p\le q<\infty$ with $1/q=1/p-\alpha/n$, $T_{\alpha}$ be the integral operator associated with $\Omega$ and $\alpha$. Suppose $\Omega\in L^\infty(\mathbb{S}^{n-1})$ and satisfies (\ref{mean value zero}) for $\alpha=0$, $\omega\in A_{p,q}$. Then
$$b\in BMO(\mathbb{R}^n)\,\Rightarrow \, [b,T_{\alpha}]: \, L^p(\omega^p)\rightarrow\, L^q(\omega^{q}).$$
}

{\bf Theorem E} (cf. \cite{Chaf-Cruz-2017}).\quad{\it (i) Let $1<p<\infty$, $T_0$ be the singular integral operator associated to $\Omega$ with the homogenous of degree $0$. Suppose $\Omega\in C^\infty(\mathbb{S}^{n-1})$ satisfying (\ref{mean value zero}), $b\in \bigcup_{q>}L_{\rm loc}^q(\mathbb{R}^n)$, $\omega\in A_p$. Then
$$ [b,T_0]: \, L^p(\omega)\rightarrow\, L^p(\omega)\,\Rightarrow \,b\in BMO(\mathbb{R}^n).$$

(ii)\quad Let $0<\alpha<n$, $1<p<q<\infty$ with $1/q=1/p-\alpha/n$, and $I_\alpha$ be the Riesz potential in $\mathbb{R}^n$. Suppose that $\omega\in A_{p,q}$ and $b\in \bigcup_{q>}L_{\rm loc}^q(\mathbb{R}^n)$. Then
$$ [b,I_\alpha]: \, L^p(\omega^p)\rightarrow\, L^p(\omega^{q})\,\Rightarrow \,b\in BMO(\mathbb{R}^n).$$
}
Applying our general theorems in Section 2, we can deduce the following result.

\medskip

\begin{corollary}\label{corollary,one-weight}
  Let $0\leq \alpha<n$, $1<p\le q <\infty$ with $1/q=1/p-\alpha/n$, $\omega^p$ be a doubling weight, and $T_{\alpha}$ be the integral operator associated to the kernel $\Omega$ with nonzero homogeneous function of degree 0 and satisfying (\ref{mean value zero}) for $\alpha=0$. Suppose that there exists an open subset $E\subset\mathbb{S}^{n-1}$ such that
  \begin{equation}\label{one-weight}
    \lim_{r\rightarrow 0}\frac{1}{\sigma(B(x',r)\cap \mathbb{S}^{n-1})}\int_{B(x',r)\cap \mathbb{S}^{n-1}}|\Omega(z')-\Omega(x')|d\sigma(z')=0
  \end{equation}
  uniformly for all $x'\in E$.
For $b\in L_{\rm loc}^1(\mathbb{R}^n)$, if the commutator $[b,T_{\alpha}]$ is bounded from $L^p(\omega^p)$ to $L^q(\omega^{q})$, then $b\in BMO(\mathbb{R}^n)$.
\end{corollary}

\begin{proof}
  Take $X=L^p(\omega^p)$, $Y=\widetilde{Y}=L^q(\omega^{q})$, $Z=L^{\infty}$, $\mu(Q)=|Q|$.
  Using H\"{o}lder's inequality, we have
  \begin{equation*}
    \left(\int_{Q}\omega^p(x)dx\right)^{1/p}
    \leq |Q|^{\alpha/n}\left(\int_{Q}\omega(x)^{q}dx\right)^{1/q}.
  \end{equation*}
  Thus, $\|\chi_{Q}\|_X\lesssim \|\chi_{Q}\|_Y\mu(Q)|Q|^{\alpha/n-1}$ for all cubes $Q$.
  The final conclusion follows immediately from
  Theorem \ref{theorem, structure, homogeneous kernel} and Proposition \ref{proposition, technique, homogeneous kernel, Z=L^infty}.
\end{proof}

\medskip

Note that $C(\mathbb{S}^{n-1})\subset L^\infty(\mathbb{S}^{n-1})$ and if $\Omega\in C(\mathbb{S}^{n-1})$, then $\Omega$ satisfies (\ref{one-weight}).
Invoking Theorem D and Corollary \ref{corollary,one-weight}, we obtain the following characterization theorem.

\begin{corollary}
Let $0\le \alpha<n$, $1<p\le q<\infty$ with $1/q=1/p-\alpha/n$, $T_{\alpha}$ be the integral operator associated to $\Omega$ with the homogenous of degree $0$. Suppose $\Omega\in C(\mathbb{S}^{n-1})$ and satisfies (\ref{mean value zero}) for $\alpha=0$, $\omega^{1/p}\in A_{p,q}$. Then
$$b\in BMO(\mathbb{R}^n)\,\Leftrightarrow \, [b,T_{\alpha}]: \, L^p(\omega)\rightarrow\, L^q(\omega^{q/p}).$$
\end{corollary}

Moreover, the following known weak-type endpoint estimate was established by Perez in 1995 (see \cite[Theorem 8.1]{Perez-1995}, and also \cite[Corollary 1]{Ort}).

\medskip

{\bf Theorem F} (\cite{Perez-1995,Ort}).\quad{\it Let $T_0$ be the singular integral operator associated with $\Omega$
satisfying (\ref{mean value zero}).
Suppose $\Omega\in {\rm Lip}_1(\mathbb{S}^{n-1})$, $\omega\in A_1$, $\Phi(t)=t(1+\log^+t)$. Then
$$b\in BMO(\mathbb{R}^n)\,\Longrightarrow\,\omega(\{x\in\mathbb{R}^n:\,|[b,T_0](f)(x)|>\lambda\})\lesssim \int_{\mathbb{R}^n}\Phi(\frac{|f(y)|}{\lambda})\omega(y)dy,\quad \forall\, \lambda>0.$$
}

\medskip

Applying Propositions 2.2 and 2.3, we can establish the inverse direct result, and then obtain the following characterized theorem, which is new even for $\omega\equiv 1$ (see Corollary \ref{cor.4.2}).

\begin{corollary}\label{corollary,weak-type-endpoint}
Let $1<p<\infty$, $\Phi(t)=t(1+\log^+t)$, $\omega\in A_1$. Suppose $\Omega\in {\rm Lip}(\mathbb{S}^{n-1})$ and $b\in L_{\rm loc}^1(\mathbb{R}^n)$. Then the following two statements are equivalent:
\begin{enumerate}[(1)]
  \item $b\in BMO(\mathbb{R}^n)$.
  \item  $\omega(\{x\in \mathbb{R}^n: |[b,T_0](f)(x)|>\lambda\})\le A\displaystyle\int_{\mathbb{R}^n}\Phi\left(\frac{|f(y)|}{\lambda}\right)\omega(y)dy$
      \quad $\forall\,\lambda>0$,  where $A$ is a positive constant independent of $f$ and $\lambda$.
\end{enumerate}
\end{corollary}

\begin{proof} By Theorem F, we need only to verify $(2)\Rightarrow (1)$.  Using the assumption $\Omega\in {\rm Lip}(\mathbb{S}^{n-1})$,
we have
  \begin{equation*}
    \lim_{r\rightarrow 0}\frac{1}{\sigma(B(x',r)\cap \mathbb{S}^{n-1})}\int_{B(x',r)\cap \mathbb{S}^{n-1}}|\Omega(z')-\Omega(x')|d\sigma(z')=0
  \end{equation*}
  uniformly for all $x'\in \mathbb{S}^{n-1}$.
Then Proposition \ref{proposition, technique, homogeneous kernel, Z=L^infty} implies that
  \begin{equation}\label{for proof, application, 1}
    \lim\limits_{\substack{|h|\rightarrow \infty}}
     \left\|\int_{Q_0}|\Omega(\cdot-y+h)-\Omega(\cdot+h)|dy\right\|_{L^{\infty}(Q_0)}\rightarrow 0.
  \end{equation}
Again, since $\Omega\in {\rm Lip}(\mathbb{S}^{n-1})$,  without loss of generality we assume that
$c\leqslant \Omega(x')\leqslant C$ for some
  open set $E\subset \mathbb{S}^{n-1}$ where $c<C<0$.
Then, invoking Proposition \ref{proposition, technique, homogeneous kernel, pointwise} for $\alpha=0$,
we deduce that for any cube $Q_1\subset \mathbb{R}^n$, there exist two functions $\phi$ and $\psi$ satisfying
  $|\phi|,|\psi|\leq 2\chi_{Q_1}$, and a cube $Q$ with same side length of $Q_1$ and $Q_1\subset \widetilde{\lambda} Q$ for some $\widetilde{\lambda}>0$
  independent of $Q_1$, such that
  \begin{equation*}
 B:=\frac{1}{|Q_1|}\int_{Q_1}|b(y)-b_{Q_1}|dy
 \leq \bar{C}
 \left(|[b,T](\phi)(x)|+|[b,T](\psi)(x)|\right)\ \text{for\ all\ }x\in Q.
\end{equation*}
Take $\lambda=\frac{\bar{C}^{-1}B}{4}$. Then
  \begin{equation*}
  \begin{split}
    \omega(Q)= &\omega(\{x\in Q: \bar{C}^{-1}B>2\lambda\})
        \leqslant  \omega(\{x\in Q: |[b,T](\phi)(x)|+|[b,T](\psi)(x)|>2\lambda\})
    \\
    \leqslant &
    \omega(\{x\in Q: |[b,T](\phi)(x)|>\lambda\})
    +
    \omega(\{x\in Q: |[b,T](\psi)(x)|>\lambda\}).
  \end{split}
  \end{equation*}
  Invoking (2) yields that
  \begin{equation*}
    \begin{split}
      \omega(\{x\in Q: |[b,T](\phi)(x)|>\lambda\}) \leqslant
      A\int_{\mathbb{R}^n}\Phi\left(\frac{|\phi(y)|}{\lambda}\right)\omega(y)dy
          \leqslant
      A\int_{Q_1}\Phi(\frac{2}{\lambda})\omega(y)dy
      =
      A\,\Phi(\frac{2}{\lambda})\omega(Q_1).
    \end{split}
  \end{equation*}
  Similarly,
  \begin{equation*}
      \omega(\{x\in Q: |[b,T](\psi)(x)|>\lambda\})
      \leqslant
      A\,\Phi(\frac{2}{\lambda})\omega(Q_1).
  \end{equation*}
Combining with the above estimates, we conclude that
\begin{equation*}
  \begin{split}
    \omega(Q) \leq 2A\,\Phi(\frac{2}{\lambda})\omega(Q_1)
          = 2A\,\Phi(\frac{8\bar{C}}{B})\omega(Q_1).
  \end{split}
\end{equation*}
Note that $Q_1\subset \widetilde{\lambda} Q$ and the doubling property of $\omega\in A_1$, we get
\begin{equation*}
  \begin{split}
    \omega(Q) \lesssim 2A\,\Phi(\frac{8\bar{C}}{B})\omega(Q_1)
    \leq 2A\,\Phi(\frac{8\bar{C}}{B})\omega( \widetilde{\lambda}Q)\lesssim A\,\Phi(\frac{8\bar{C}}{B})\omega(Q).
  \end{split}
\end{equation*}
This implies
\begin{equation*}
  B\lesssim \max\{1, A\},
\end{equation*}
 and complete the proof of Corollary \ref{corollary,weak-type-endpoint}.
 \end{proof}

\subsubsection{Weighted $BMO$ and commutators in two-weight setting}

\medskip

In 1985, Bloom \cite{Bl} gave the characterization of $BMO_\mu$, the weighted BMO space, via the $(L^p(\omega),\, L^p(\lambda))$-boundedness of $[b, H]$, where $H$ is the Hilbert transform, $\omega,\,\lambda\in A_p$, $\mu=(\omega/\lambda)^{1/p}$, $1<p<\infty$. Recently, for the commutators of singular integrals in higher dimension, Holmes et al. \cite{Holmes-Lacey-Wick_MathA_2017} obtained the following results.

\medskip

{\bf Theorem G} (\cite{Holmes-Lacey-Wick_MathA_2017}). Let $1<p<\infty$, $\omega,\,\lambda\in A_p$, $\mu=(\omega/\lambda)^{1/p}$. Then,

\medskip

(i)\quad for the $j$-th Riesz transform $R_j$ in $\mathbb{R}^n$,
   $$b\in BMO_\mu\,\Leftrightarrow \,[b,R_j]:\, L^p(\omega)\rightarrow\, L^p(\lambda),\quad\text{for some}\, j=1,\cdots,n;$$

(ii)\quad  for the general Calderon-Zygmund operator $T$ in $\mathbb{R}^n$,
 $$b\in BMO_\mu\,\Rightarrow\, [b, T]:\, L^p(\omega)\rightarrow\, L^p(\lambda).$$

\medskip

Invoking Theorem 2.1 and Proposition 2.3, we generalize and complete the above results as follows.

\begin{corollary}\label{cor.4.8}
  Let $1<p<\infty$, $\omega,\lambda\in A_p$ $\mu=(\omega/\lambda)^{1/p}$, and $T_0$ be the singular integral operator
  associated with $\Omega$. Suppose $\Omega\in C(\mathbb{S}^{n-1})$ satisfying (\ref{mean value zero}).
  Then the following two statements are equivalent.
\begin{enumerate}[(1)]
  \item $b\in BMO_{\mu}$.
  \item $[b,T_{0}]$ is bounded from $L^p(\omega)$ to $L^p(\lambda)$.
\end{enumerate}
Moreover, we have $\|b\|_{BMO_{\mu}}\sim \|[b,T_0]\|_{L^p(\omega)\rightarrow L^p(\lambda)}$.
\end{corollary}

\begin{remark} We remark that after our results are showed, Lerner, Ombrosi and Rivera-R\'ios \cite{LOR} obtained very recently the same result provided $\Omega$ satisfies the Dini condition in different way. In fact, "$(2)\Longrightarrow (1)$" was given in a more general form in \cite{LOR}
\end{remark}

\begin{proof}
  By the similar arguments to that in  \cite{Ding-1999}, it is not hard to verify that if $\Omega\in L^\infty(\mathbb{S}^{n-1})$ and $b\in BMO_\mu$, then
  $[b,T_{0}]$ is bounded from $L^p(\omega)$ to $L^p(\lambda)$. Noting that $C(\mathbb{S}^{n-1})\subset L^\infty(\mathbb{S}^{n-1})$,
  we get that $(1)\Longrightarrow (2)$.

  To deal with $(2)\Longrightarrow (1)$, we take $\alpha=0$, $X=L^p(\omega)$ $Y=\widetilde{Y}=L^p(\lambda)$, $Z=L^{\infty}$.
  By Theorem \ref{theorem, structure, homogeneous kernel} and Proposition 2.3, we only need to verify
  \begin{equation*}
    \|\chi_Q\|_X\lesssim \|\chi_Q\|_Y\mu(Q)|Q|^{\alpha/n-1}= \|\chi_Q\|_Y\mu(Q)|Q|^{-1},
  \end{equation*}
  which is equivalent with
    \begin{equation}\label{for proof, application, 2}
    |Q|\omega(Q)^{1/p}\lesssim \lambda(Q)^{1/p}\mu(Q).
  \end{equation}
  Denote $A=p+1+\frac{1}{p'-1}$, then $1=\frac{1}{A/p}+\frac{1}{A}+\frac{1}{(p'-1)A}$.
  Recalling $\mu=(\omega/\lambda)^{1/p}$, we use H\"{o}lder's inequality to deduce that
  \begin{equation*}
    \begin{split}
      |Q|
      = &
      \int_{Q}\mu^{p/A}(x)\lambda^{1/A}(x)\omega^{-1/A}(x)dx
      \\
      \lesssim &
      \left(\int_{Q}\mu(x) dx\right)^{p/A}
      \left(\int_{Q}\lambda(x) dx\right)^{1/A}
      \left(\int_{Q}\omega^{1-p'}(x)dx\right)^{\frac{1}{(p'-1)A}}.
      \\
      \sim &
      \mu(Q)^{p/A}\lambda(Q)^{1/A}\left(\frac{1}{|Q|}\int_{Q}\omega^{1-p'}(x)dx\right)^{\frac{1}{(p'-1)A}}|Q|^{\frac{1}{(p'-1)A}}.
    \end{split}
  \end{equation*}
  Since $\omega\in A_p$, we have
  \begin{equation*}
    \left(\frac{1}{|Q|}\int_Q\omega(x)dx\right)\left(\frac{1}{|Q|}\int_Q\omega(x)^{1-p'}dx\right)^{p-1}\lesssim 1,
  \end{equation*}
  which implies that
  \begin{equation*}
    \left(\frac{1}{|Q|}\int_{Q}\omega^{1-p'}(x)dx\right)^{\frac{1}{(p'-1)A}}
    \lesssim
    \left(\frac{1}{|Q|}\int_Q\omega(x)dx\right)^{\frac{1}{(1-p)(p'-1)A}}
    =
    \left(\frac{1}{|Q|}\int_Q\omega(x)dx\right)^{\frac{-1}{A}}.
  \end{equation*}
Consequently,
    \begin{equation*}
    \begin{split}
      |Q|
      \lesssim &
      \mu(Q)^{p/A}\lambda(Q)^{1/A}\left(\frac{1}{|Q|}\int_Q\omega(x)dx\right)^{\frac{-1}{A}}|Q|^{\frac{1}{(p'-1)A}},
    \end{split}
  \end{equation*}
that is,
      \begin{equation*}
    \begin{split}
      |Q|^{p/A}\omega(Q)^{1/A}
      \lesssim &
      \mu(Q)^{p/A}\lambda(Q)^{1/A}.
    \end{split}
  \end{equation*}
 This implies the desired conclusion and completes the proof of Corollary \ref{cor.4.8}.
\end{proof}

On the other hand, for the commutator of fractional integral operators $[b, T_\alpha]$, Ding and Lu \cite{Ding-1999} proved that for $0<\alpha<n$, $1<p<n/\alpha$, $1/q=1/p-\alpha/n$, $\omega,\,\lambda\in A_{p,q} $, $\mu=\omega/\lambda$, if $\Omega\in L^{\infty}(\mathbb{S}^{n-1})$, then
\begin{equation}\label{l4.2.2}
b\in {\rm BMO}_\mu\,\Longrightarrow\, [b,T_\alpha]:\,L^p(\omega^p)\to L^q(\lambda^q).
\end{equation}
In particular, for the commutators of Riesz potential, Holmes et al. \cite{HLS} recently showed that
\begin{equation}\label{l4.2.3}
b\in {\rm BMO}_\mu\,\Longleftrightarrow\, [b,I_\alpha]:\,L^p(\omega^p)\to L^q(\lambda^q).
\end{equation}

Invoking Theorem 2.1 and Proposition 2.3 again, by (\ref{l4.2.2}) and the similar arguments to in the proof of Corollary \ref{cor.4.8}, we can get the following result, which essentially improve and extend the result of \cite{HLS} and present the converse result of (\ref{l4.2.2}).

\begin{corollary}
 Let $0<\alpha<n$, $1<p<n/\alpha$, $1/q=1/p-\alpha/n$, $\omega,\,\lambda\in A_{p,q} $, $\mu=\omega/\lambda$. If $\Omega\in C(\mathbb{S}^{n-1})$, then
\begin{equation}\label{l4.2.4}
b\in {\rm BMO}_\mu\,\Longleftrightarrow\, [b,T_\alpha]:\,L^p(\omega^p)\to L^q(\lambda^q).
\end{equation}
\end{corollary}

\subsection{$BMO$ and Commutators in Multilinear Setting}

\medskip

After the pioneering work of Grafakos-Torres \cite{Grafakos-Torres_Adv.Math_2002}
in the multilinear Calder\v{o}n-Zyamund theory, P\'{e}rez-Torres \cite{PT-2003}
first introduced the $i$-th commutator of m-linear Calder\v{o}n-Zyamund operator $T$ and showed that $[b, T_{K_0}]_i$ is bounded from $L^{p_1}\times\cdots\times L^{p_m}$ to $L^p$ provided that $b\in BMO(\mathbb{R}^n)$, $1<p_1,\cdots,p_m,p<\infty$ with $1/p=1/p_1+\cdots+1/p_m$. Subsequently, in the celebrated work \cite{LOPTT}  Lerner et al. removed the restriction of that $p>1$ and established the multiple weighted version as well as the weak-type endpoint estimate, and see \cite{An-Duong, ChenWu1} for the non-smooth kernels cases. The corresponding results for the commutators of multilinear fractional integrals were given by Chen and Xue \cite{Chen-Xue} (see also \cite{Yu-Chen, ChenWu2}).
Recently, Chaffee \cite{Chaf} obtained the following characterized theorem.

\medskip

{\bf Theorem H} (cf. \cite{Chaf})\quad (i)\quad Suppose that $K_0$ is a homogeneous function of degree $-mn$, and there exists a ball $B\subset \mathbb{R}^{mn}$ such that $1/K_0$ can be expended to a Fourier series in $B$. If $1<p_1,\cdots,p_m,{p}<\infty$, $\frac 1p=\frac 1{p_1}+\cdots+\frac 1{p_m}$,
$$b\in BMO(\mathbb{R}^n)\,\Longleftrightarrow\, [b,T_{K_0}]_j:\, L^{p_1}(\mathbb{R}^n)\times\cdots\times L^{p_m}(\mathbb{R}^n)\to L^p(\mathbb{R}^n);$$

\bigskip

(ii)\quad Let $0<\alpha<mn$, $1<p_1,\cdots,p_m,{ p}<\infty$ with $\frac 1p=\frac 1{p_1}+\cdots+\frac1{p_m}$, $\frac 1q=\frac 1p-\frac{\alpha}n$. Then
$$b\in BMO(\mathbb{R}^n)\,\Longleftrightarrow\, [b,I_{\alpha,m}]_j:\, L^{p_1}(\mathbb{R}^n)\times\cdots\times L^{p_m}(\mathbb{R}^n)\to L^q(\mathbb{R}^n).$$

\medskip

The conclusion (i) in Theorem H was also obtained by Li and Wick \cite{LW} in a different way. Furthermore, applying Theorem B, Chaffee and Cruz-Uribe \cite{Chaf-Cruz-2017} established the following weighted version of Theorem H.

\medskip

{\bf Theorem I} (cf. \cite{Chaf-Cruz-2017})\quad (i)\quad Suppose that $K_0$ is a homogeneous function of degree $-mn$, and there exists a ball $B\subset \mathbb{R}^{mn}$ such that $1/K_0$ can be expended to a Fourier series in $B$. If $1<p_1,\cdots,p_m,{p}<\infty$ with $\frac 1p=\frac 1{p_1}+\cdots+\frac 1{p_m}$, $\omega_i\in A_{p_i}$ ($i=1,\cdots,m$), $\mu_{\vec{\omega}}=\prod_{i=1}^m\omega_i^{p/p_i}\in A_p$, then
$$b\in BMO(\mathbb{R}^n)\,\Longleftrightarrow\, [b,T_{K_0}]_j:\, L^{p_1}(\omega_1)\times\cdots\times L^{p_m}(\omega_m)\to L^p(\mu_{\vec{\omega}});$$

\medskip

(ii)\quad Let $0<\alpha<mn$, $1<p_1,\cdots,p_m,{p}<\infty$ with $\frac 1p=\frac 1{p_1}+\cdots+\frac 1{p_m}$, $\frac 1{q_j}=\frac 1{p_j}-\frac{\alpha}{mn}$, and $\frac 1q=\frac 1{q_1}+\cdots+\frac 1{q_m}$. Suppose that $\omega_j\in A_{p_i,q_i}$ and $\nu_{\vec{\omega}}=\prod_{i=1}^m\omega_i\in A_q$. Then
$$b\in BMO(\mathbb{R}^n)\,\Longleftrightarrow\,[b,I_{\alpha,m}]_j:\, L^{p_1}(\omega_1^{p_1})\times\cdots\times L^{p_m}(\omega_m^{p_m})\to L^q(\nu_{\vec{\omega}}).$$

\medskip

\begin{remark}In \cite{Chaf-Cruz-2017}, Chaffee and Cruz-Uribe pointed out that their arguments are not valid for $p<1$, and the weights are restricted in a narrower class.

Our next corollaries will remove the restriction of that $p>1$, weaken the condition of kernel $K_0$ and enlarge the weights class, moreover, present a characterized result of weak-type endpoint estimate, which is new even in un-weighted case.
\end{remark}

\begin{corollary}\label{corollary, application, multilinear}
Let $m\in \mathbb{N}$ and $T_K$ be a $m$-linear Calderon-Zygmund operator with kernel $K$ satisfying following lower bound property: there exist a positive constant $c$ and an open cone $\widetilde{\Gamma}\subset\mathbb{R}^{mn}$ with vertex at the origin such that
\begin{equation*}
  |K(x,y_1,\cdots,y_m)|\geq \frac{c}{(\sum_{j=1}^m|x-y_j|)^{mn}},\qquad \forall\, (x-y_1,\cdots,x-y_m)\in \widetilde{\Gamma}.
\end{equation*}
Let $b\in L_{loc}^1(\mathbb{R}^n)$ and $[b,T_K]_i$ be the $i$-th commutator generated by $T_K$ with $b$, $1\le i\le m$. Then the following three statements are equivalent:
\begin{enumerate}
  \item $b\in BMO(\mathbb{R}^n)$.

 \item For $\omega_j\in A_1$, $j=1,\cdots,m$, $v_{\vec{\omega}}=\prod_{j=1}^m\omega_j^{1/m}$, $\Phi(t)=t(1+\log^+t)$,
  $$v_{\vec{\omega}}(\{x\in \mathbb{R}^n: |[b,T_K]_i(f)(x)|>\lambda^m\})\lesssim \prod_{j=1}^m\left(\int_{\mathbb{R}^n}\Phi(\frac{|f(y)|}{\lambda})\omega_j(y)dy\right)^{{1}/{m}}$$
  for any $\lambda>0$.

 \item For $1<p_j<\infty$, $\omega_j\in A_{p_j}$, $j=1,\cdots, m$, $v_{\vec{\omega}}=\prod_{j=1}^m\omega_j^{p/p_j}$ with $1/p=1/p_1+\cdots+1/p_m$, $$\|[b,T_K]_i(f_1,\cdots,f_m)\|_{L^p(v_{\vec{\omega}})}\lesssim \prod_{j=1}^m\|f_j\|_{L^{p_j}(\omega_j)},\qquad \forall\, f_j\in L^{p_j}(\omega_j),\, j=1,\cdots,m.$$

\end{enumerate}
\end{corollary}

\begin{proof}
  For simplicity, we only present the proof for $m=2$ and $i=1$, since the other cases can be treated similarly.
  Note that $v_{\vec{\omega}}\in A_{\vec{P}}$ for $\omega_j\in A_{p_j}$, $1\le p_j<\infty$, $j=1,\,2$. Then
  $(1)\Longrightarrow (2)$ and $(1)\Longrightarrow (3)$ follow directly from \cite[Theorems 3.16 and 3.18]{LOPTT}.

  Next, we deal with the opposite direction. Without loss of generality, using the lower bound of $K$ and the property of Calder\'{o}n-Zygmund kernel, we have
    \begin{equation*}
    \frac{c}{(|x-y|+|x-z|)^{2n}}\leqslant K(x,y,z) \leqslant \frac{C}{(|x-y|+|x-z|)^{2n}}
  \end{equation*}
  for all $(x-y,x-z)\in \widetilde{\Gamma}$, where $0<c<C$.
 By Propositions \ref{proposition, technique, standard kernel, Z=L^infty}
  and \ref{proposition, technique, standard kernel, pointwise} for $\alpha=0,\,i=1$, for any fixed cube $Q_1$,
we can find $\phi_j,\,\psi_j$ satisfying
  $|\phi_j|,\,|\psi_j|\leq 2\chi_{Q_j}\, (j=1,\,2)$, where $|Q_1|=|Q_2|$,
  and find a cube $Q$ with the same side length of $Q_1$  such that $Q_1,Q_2\subset \lambda Q$ for some $\lambda>0$ independent of $Q_1$,
  and
\begin{equation*}
  \begin{split}
    B:= &\frac{1}{|Q_1|}\int_{Q_1}|b(y)-b_{Q_1}|dy\\
    \leq &\widetilde{C}
    \left(|[b,T_K]_1(\phi_1,\phi_2)(x)|+|[b,T_K]_1(\psi_1,\psi_2)(x)|\right)\quad \text{for\ all\ }x\in Q,
  \end{split}
\end{equation*}
where the constant $\widetilde{C}$ is independent of the choice of $Q_1$.
Taking $\lambda=\sqrt{\frac{\tilde{C}^{-1}B}{4}}$, we get
  \begin{equation*}
  \begin{split}
    v_{\vec{\omega}}(Q)= &v_{\vec{\omega}}(\{x\in Q: \widetilde{C}^{-1}B>2\lambda^2\})\\
        \leqslant&  v_{\vec{\omega}}(\{x\in Q: |[b,T_K]_1(\phi_1,\phi_2)(x)|+|[b,T_K]_1(\psi_1,\psi_2)(x)|>2\lambda^2\})
    \\
    \leqslant &
    v_{\vec{\omega}}(\{x\in Q: |[b,T_K]_1(\phi_1,\phi_2)(x)|>\lambda^2\})
    +
    v_{\vec{\omega}}(\{x\in Q: |[b,T_K]_1(\psi_1,\psi_2)(x)|>\lambda^2\}).
  \end{split}
  \end{equation*}
  Using (2), we deduce that
  \begin{equation*}
    \begin{split}
     & v_{\vec{\omega}}(\{x\in Q: |[b,T_K]_1(\phi_1,\phi_2)(x)|>\lambda^2\})\\
      &\qquad\qquad\lesssim
      \left(\int_{\mathbb{R}^n}\Phi(\frac{|\phi_1(y)|}{\lambda})\omega_1(y)dy\right)^{1/2}
      \left(\int_{\mathbb{R}^n}\Phi(\frac{|\phi_2(y)|}{\lambda})\omega_2(y)dy\right)^{1/2}
      \\
     &\qquad\qquad \lesssim
      \left(\int_{Q_1}\Phi(\frac{2}{\lambda})\omega_1(y)dy\right)^{1/2}\left(\int_{Q_2}\Phi(\frac{2}{\lambda})\omega_2(y)dy\right)^{1/2}\\
     &\qquad\qquad \leq
      \Phi(\frac{2}{\lambda})\omega_1(Q_1)^{1/2}\omega_2(Q_2)^{1/2}.
    \end{split}
  \end{equation*}
  Similarly,
  \begin{equation*}
      v_{\vec{\omega}}(\{x\in Q: |[b,T_K]_1(\psi_1,\psi_2)(x)|>\lambda^2\})
      \lesssim
      \Phi(\frac{2}{\lambda})\omega_1(Q_1)^{1/2}\omega_2(Q_2)^{1/2}.
  \end{equation*}
Combining with the above estimates, we conclude that
\begin{equation*}
  \begin{split}
    v_{\vec{\omega}}(Q)
    \lesssim
    \Phi(\frac{2}{\lambda})\omega_1(Q_1)^{1/2}\omega_2(Q_2)^{1/2}
      =
    \Phi(\,\sqrt{{16\widetilde{C}}/{B}}\,)\,\omega_1(Q_1)^{1/2}\,\omega_2(Q_2)^{1/2}.
  \end{split}
\end{equation*}
Sine $\omega_1,\,\omega_2\in A_1$, we have $\omega_i(Q)\lesssim |Q|\inf_Q\omega_i$, $i=1,\,2$. Then
\begin{equation*}
  \begin{split}
    \omega_1(Q)^{1/2}\omega_2(Q)^{1/2}
    \lesssim
    |Q|\left(\inf_Q\omega_1\inf_Q\omega_2\right)^{1/2}
        \lesssim
    \int_Q\omega_1(x)^{1/2}\omega_2(x)^{1/2}dx=v_{\vec{\omega}}(Q).
  \end{split}
\end{equation*}
Thus, we use the fact $Q_1,Q_2\subset \lambda Q$ and the doubling property of $\omega_1,\omega_2$ to deduce that
\begin{equation*}
  \omega_1(Q_1)^{1/2}\omega_2(Q_2)^{1/2}
  \lesssim
  \omega_1(Q)^{1/2}\omega_2(Q)^{1/2}
  \lesssim
  v_{\vec{\omega}}(Q).
\end{equation*}
This implies that
\begin{equation*}
  1\lesssim \Phi(\sqrt{{16\widetilde{C}}/{B}}),
\end{equation*}
and completes the proof of $(2)\Longrightarrow (1)$.

Finally, we verify that $(3)\Longrightarrow (1)$. Take $X_j=L^{p_j}(\omega_j)$ ($j=1,\,2$), $Y=\widetilde{Y}=L^p(v_{\widetilde{\omega}})$, $Z=L^{\infty}$, $\alpha=0$, $\mu(Q)=|Q|$. By Theorem 3.1, it suffices to show that
\begin{equation*}
  \|\chi_Q\|_{X_1}\cdot \|\chi_Q\|_{X_2}\lesssim \|\chi_Q\|_{Y}\ \text{for every cube}\ Q.
\end{equation*}
Note that $1=\frac{1}{2p_1'}+\frac{1}{2p_2'}+\frac{1}{2p}$. The H\"{o}lder inequality yields that
\begin{equation*}
  \begin{split}
  |Q| &=  \int_{Q}\omega_1(x)^{\frac{1}{2p_1'}-\frac{1}{2}}\omega_2(x)^{\frac{1}{2p_2'}-\frac{1}{2}}v_{\vec{\omega}}(x)^{\frac{1}{2p}} dx\\
   &\leq
  \left(\int_{Q} \omega_1^{1-p_1'}(x)dx\right)^{\frac{1}{2p_1'}}\left(\int_{Q} \omega_2^{1-p_2'}(x)dx\right)^{\frac{1}{2p_2'}}
  \left(\int_Qv_{\vec{\omega}}(x)dx\right)^{1/2p}.
  \end{split}
\end{equation*}
Using the property of $A_p$ weight, we have
  \begin{equation*}
    \left(\int_Q\omega_j(x)dx\right)^{1/p_j}\left(\int_Q\omega_j(x)^{1-p_j'}dx\right)^{1/p_j'}\lesssim |Q|.
  \end{equation*}
Then
\begin{equation*}
  \begin{split}
    \prod_{j=1}^2\left(\int_Q\omega_j(x)dx\right)^{1/p_j}\left(\int_Q\omega_j(x)^{1-p_j'}dx\right)^{1/p_j'}
      \lesssim
    |Q|^2
    \lesssim
    \prod_{j=1}^2\left(\int_Q\omega_j(x)^{1-p_j'}dx\right)^{1/p_j'}\left(\int_Qv_{\vec{\omega}}(x)dx\right)^{1/p},
  \end{split}
\end{equation*}
which implies that
\begin{equation*}
  \prod_{j=1}^2\left(\int_Q\omega_i(x)dx\right)^{1/p_i}\lesssim \left(\int_Qv_{\vec{\omega}}(x)dx\right)^{1/p},
\end{equation*}
that is,
\begin{equation*}
  \|\chi_Q\|_{X_1}\cdot \|\chi_Q\|_{X_2}\lesssim \|\chi_Q\|_{Y}\ \text{for every cube}\ Q.
\end{equation*}
This completes the proof of  that $(3)\,\Longrightarrow (1)$.
\end{proof}

\begin{remark}
Using \cite[Theorem 1.6]{Ofs-Moen_Arxiv_2017}, the weight condition (3) in Corollary \ref{corollary, application, multilinear}
 can be replaced by
$\vec{\omega}=(\omega_1,\cdots,\omega_m)\in  A_{\vec{P}}$ with $\omega_j\in A_{\infty}$ for $1\leq j\leq m$.
\end{remark}

Similarly, invoking \cite[Theorem 1.4]{ChenWu2} and applying
Propositions \ref{Theorem, structure, standard kernel}
and \ref{proposition, technique, standard kernel, Z=L^infty},
for the commutators of multilinear fractional integrals we can obtain the following result.

\begin{corollary}
Let $m,\,n\in \mathbb{N}$, $0<\alpha<mn$, $1<p_1,\cdots,p_m<\infty$, $\frac 1p=\frac 1{p_1}+\cdots+\frac 1{p_m}$, $\frac 1{q}=\frac 1{p}-\frac{\alpha}{n}$,
$\vec{\omega}=(\omega_1,\cdots,\omega_m)\in A_{\vec{P},q}$, $\omega_j^{p_j}\in A_\infty$ ($j=1,\cdots,m$), $\nu_{\vec{\omega}}=\prod_{j=1}^m\omega_j$. Then for any fixed $i\in\{1,\cdots,m\}$,
$$b\in BMO(\mathbb{R}^n)\,\Longleftrightarrow\,[b,I_{\alpha,m}]_i:\, L^{p_1}(\omega_1^{p_1})\times\cdots\times L^{p_m}(\omega_m^{p_m})\to L^q(\nu_{\vec{\omega}}^q).$$
\end{corollary}

\subsection{Lipschitz Function Spaces and Commutators}

\medskip

The investigation on the characterization of Lipschitz spaces via the boundedness of commutators in certain function spaces has also attracted a number of attentions. In this subsection, we will apply our general theorem to present several new developments in this topic.

\subsubsection{${\rm Lip}_\beta(\mathbb{R}^n)$ and commutators}

\medskip

In 1978, Janson \cite{Jan} first proved that for $0<\beta<1$, $b\in {\rm Lip}_\beta(\mathbb{R}^n)$ if and only if $[b,T_0]$ is bounded from $L^p(\mathbb{R}^n)$ to $L^q(\mathbb{R}^n)$ for $1<p<q<\infty$ with $1/q=1/p-\beta/n$. Later on, Paluszynski \cite{Pal-1995} established the corresponding result for the commutator of Riesz potential $[b,I_\alpha]$, and extended these results to the boundedness of $(L^p, \dot{F}_q^{\beta,\infty})$, where $ \dot{F}_q^{\beta,\infty}$ denotes the homogeneous Triebel-Lizorkin spaces (see \cite{Pal-1995}).

In 1965, Calder\'on \cite{Calderon-1965} showed that for $1<p<\infty$ and $\Omega\in L(\log^+L)(\mathbb{S}^{n-1})$ satisfying that
\begin{equation}\label{l4.6}
\int_{\mathbb{S}^{n-1}}\Omega(x')x_j'd\sigma(x')=0,\, j=1,\cdots,n.
\end{equation}
Then
\begin{equation}\label{l4.5}
b\in {\rm Lip}(\mathbb{R}^n)\,\Longrightarrow\, [b,T_{-1}]:\,L^p(\mathbb{R}^n)\to L^p(\mathbb{R}^n),
\end{equation}
and $\|[b,T_{-1}]\|_{L^p(\mathbb{R}^n)\rightarrow L^p(\mathbb{R}^n)}\lesssim \|b\|_{{\rm Lip}(\mathbb{R}^n)}$. Recently, Chen, Ding and Hong \cite{CDH-2016} obtained the following result.

\medskip

{\bf Theorem J} (cf. \cite{CDH-2016})\quad Let $1<p<\infty$. Suppose that $b\in L_{\rm loc}^1(\mathbb{R}^n)$ and $\Omega\in {\rm Lip}(\mathbb{S}^{n-1})$ satisfying (\ref{mean value zero}) and (\ref{l4.6}).
Then the following statements are equivalent:
\begin{enumerate}[(i)]
  \item $b\in {\rm Lip}(\mathbb{R}^n)$,
  \item $[b,T_{-1}]$ is bounded on $L^p(\mathbb{R}^n)$,
  \item $[b,T_{-1}]$ is bounded from $L^1(\mathbb{R}^n)$ to $L^{1,\infty}(\mathbb{R}^n)$.
\end{enumerate}

\medskip

Using Proposition \ref{proposition, structure and technique for Z=L^1}, we can weaken the condition of $\Omega$ in Theorem J as follows.

\begin{corollary}\label{corollary, calderon commutator}
  Let $1<p<\infty$. Suppose that $b\in L_{loc}^1(\mathbb{R}^n)$, $\Omega\in L(\log^+L)(\mathbb{S}^{n-1})$ satisfying (\ref{mean value zero})
  and (\ref{l4.6}). If there exist a open $E\subset \mathbb{S}^{n-1}$ and constants $c$ and $C$ such that
\begin{equation}\label{l4.7}
c\leqslant \Omega(x')\leqslant C,\,\forall\, x'\in E,\quad \text{where} \,0<c<C\,\text{or}\, c<C<0,
\end{equation}
then the following statements are equivalent:
\begin{enumerate}[(1)]
  \item $b\in {\rm Lip}(\mathbb{R}^n)$,
  \item $[b,T_{-1}]$ is bounded on $L^p(\mathbb{R}^n)$,
  \item $[b,T_{-1}]$ is bounded from $L^p(\mathbb{R}^n)$ to $L^{p,\infty}(\mathbb{R}^n)$,
\end{enumerate}

In particular, if $\Omega\in C(\mathbb{S}^{n-1})$, then (\ref{l4.7}) holds, and the three statements above are equivalent.

Furthermore, if $\Omega\in {\rm Lip}_{\beta}(\mathbb{S}^{n-1}) (0<\beta<1)$, then
\begin{equation*}
  b\in {\rm Lip}(\mathbb{R}^{n}) \Longleftrightarrow [b,T_{-1}]:L^1(\mathbb{R}^n)\rightarrow L^{1,\infty}(\mathbb{R}^n).
\end{equation*}

\end{corollary}
\begin{proof}
  $(1)\Rightarrow (2)\Rightarrow (3)$ directly follow from (\ref{l4.5}) and $L^p(\mathbb{R}^n)\subset L^{p,\infty}(\mathbb{R}^n)$.

  Now we verify $(3)\Rightarrow (1)$.
  Take $X=L^p(\mathbb{R}^n)$, $Y=L^{p,\infty}(\mathbb{R}^n)$, $\widetilde{Y}=L^{\tilde{p},1}(\mathbb{R}^n)$
  where $1/\widetilde{p}=1/p+1$, $Z=L^1(\mathbb{R}^n)$,
  and take $\alpha=-1$, $\mu(Q)=|Q|^{(n+1)/n}$.
  Recalling ${\rm Lip}(\mathbb{R}^n)=BMO_{\mu}$ with $\mu(Q)=|Q|^{(n+1)/n}$,
  this conclusion follows immediately from Proposition \ref{proposition, structure and technique for Z=L^1} and Lemma \ref{lemma, auxiliary space, Lorentz}.

  Moreover, for $\Omega\in {\rm Lip}_{\beta}(\mathbb{S}^{n-1})\, (0<
  \beta<1)$ and $b\in {\rm Lip}(\mathbb{R}^n)$, it is easy to verify that $K(x,y)=\frac{\Omega(x-y)}{|x-y|^{n+1}}(b(x)-b(y))$ is a standard kernel.
  Since $ \Omega\in {\rm Lip}_{\beta}(\mathbb{S}^{n-1})\subset L(\log^+L)(\mathbb{S}^{n-1})$, by (\ref{l4.5}) again, we get the $L^p(\mathbb{R}^n)$-  boundedness of $[b,T_{-1}]$ for $1<p<\infty$. Furthermore, the $L^1(\mathbb{R}^n)\rightarrow L^{1,\infty}(\mathbb{R}^n)$ boundedness
  of $[b,T_{-1}]$ follows from the standard arguments. For the inverse direction, we take $X=L^1(\mathbb{R}^n)$, $Y=\widetilde{Y}=L^{1,\infty}(\mathbb{R}^n)$, $Z=L^{\infty}(\mathbb{R}^n)$,
  and take $\alpha=-1$, $\mu(Q)=|Q|^{(n+1)/n}$. Then the desired conclusion follows immediately from
  Theorem \ref{theorem, structure, homogeneous kernel}, Proposition \ref{proposition, technique, homogeneous kernel, Z=L^infty}
  and Lemma \ref{lemma, auxiliary space, Lorentz}.
\end{proof}

For the generalization of characterized theorem of ${\rm Lip}_{\beta}(\mathbb{R}^n)$ ($0<\beta<1$) in \cite{Jan,Pal-1995}, we will leave it to the next subsection for more general weighted version.

\subsubsection{${\rm Lip}_{\beta,\omega}(\mathbb{R}^n)$ and commutators}

\medskip

This subsection is devoted to the characterization of the weighted Lipschitz spaces ${\rm Lip}_{\beta,\omega}$. In \cite{HG-2008}, Hu and Gu first established the following results, which can be regarded as the weighted version of the results in \cite{Jan, Pal-1995}.

\medskip

{\bf Theorem K} (cf. \cite{HG-2008})\quad{\it {\rm (i)}\quad Let $1<p,q <\infty$, $1/q=1/p-\beta/n$, $0<\beta<1$, $\omega\in A_1(\mathbb{R}^n)$.
Let $T$ be the Calderon-Zygmund operator associated with $K$. Suppose that there is a ball $B\in \mathbb{R}^n$ such that $1/K$ can be  expended an absolutely convergent Fourier series. Then for $b\in L_{\rm loc}^1(\mathbb{R}^n)$,
$$b\in {\rm Lip}_{\beta,\omega}\,\Longleftrightarrow\,[b,T]:\,L^p(\omega)(\mathbb{R}^n)\to L^q(\omega^{1-q})(\mathbb{R}^n);$$

{\rm (ii)}\quad Let $0<\alpha<n$, $0<\beta<1$ with $0<\alpha+\beta<n$, $1<p<n/(\alpha+\beta)$, $1/q=1/p-(\alpha+\beta)/n$ and $\omega\in A_1$. Then for $b\in L_{\rm loc}^1(\mathbb{R}^n)$ and the commutator of Riesz potential $[b,I_\alpha]$, we have
$$b\in {\rm Lip}_{\beta,\omega}\,\Longleftrightarrow\, [b, I_\alpha]:\, L^p(\omega)\to L^q(\omega^{1-(1-\alpha/n)q}).$$
}

Subsequently, under assuming $\Omega$ satisfies certain $L^s$-Dini conditions and $b\in {\rm Lip}_{\beta, \omega}$, Lin, Liu and Pan \cite{LLP} (resp., Liu and Zhou \cite{LZ}) proved that $[b,T_\alpha]$ (resp. $[b, T_0]$) is bounded from $L^p(\omega)$ to $L^q(\omega^{1-(1-\alpha/n)q})$ for $0<\beta<1$, $0<\alpha<n$ (resp., $\alpha=0$) with $0<\beta+\alpha<n$, $1<p<n/(\alpha+\beta)$ and $1/q=1/p-(\alpha+\beta)/n$. But it is not clear whether the corresponding converse result in \cite{LLP,LZ} is also true.

Applying our general theorem, we can establish the following result.
\begin{corollary}\label{cor.18}
  Let $1<p,q <\infty$, $1/q=1/p-(\alpha+\beta)/n$, $0<\beta<1$, $0\le\alpha<n$ with $0<\alpha+\beta<n$, $\omega$ be a doubling weight, and $T_{\alpha}$ be the integral operator associated to the kernel $\Omega$ with nonzero homogeneous function of degree 0 and satisfying (\ref{mean value zero}) for $\alpha=0$. Suppose that there exists an open subset $E\subset\mathbb{S}^{n-1}$ such that
  \begin{equation*}
    \lim_{r\rightarrow 0}\frac{1}{\sigma(B(x',r)\cap \mathbb{S}^{n-1})}\int_{B(x',r)\cap \mathbb{S}^{n-1}}|\Omega(z')-\Omega(x')|d\sigma(z')=0.
  \end{equation*}
For $b\in L_{\rm loc}^1(\mathbb{R}^n)$, we have
\begin{equation}\label{l4.9}
[b,T_{\alpha}]:\,L^p(\omega)\to L^q(\omega^{1-(1-\alpha/n)q})\,\Longrightarrow\, b\in {\rm Lip}_{\beta,\omega}.
\end{equation}
\end{corollary}
\begin{proof}
   Take $X=L^p(\omega)$, $Y=\widetilde{Y}=L^q(\omega^{1-(1-\alpha/n)q})$, $Z=L^{\infty}$, $\mu(Q)=\omega(Q)^{1+\beta/n}$.
  Observing that $1+\beta/n-1/p=1-1/q-\alpha/n\geq 0$, we have  $q(1-\alpha/n)\geq 1$.
  Thus, H\"{o}lder's inequality yields that
  \begin{equation*}
    \begin{split}
      |Q|
      = &
      \int_{\mathbb{R}^n}\omega(x)^{\frac{1}{q(1-\alpha/n)}-1}\omega(x)^{1-\frac{1}{q(1-\alpha/n)}}dx
      \\
      \leq &
      \left(\int_{\mathbb{R}^n}\omega(x)^{1-q(1-\alpha/n)}dx\right)^{\frac{1}{q(1-\alpha/n)}}
      \left(\int_{\mathbb{R}^n}\omega(x)dx\right)^{1-\frac{1}{q(1-\alpha/n)}}
      \\
      = &
       \|\chi_Q\|_{Y}^{\frac{1}{1-\alpha/n}}\omega(Q)^{1-\frac{1}{q(1-\alpha/n)}}.
    \end{split}
  \end{equation*}
  It implies $|Q|^{1-\alpha/n}\lesssim \|\chi_Q\|_{Y}\omega(Q)^{1-\alpha/n-1/q}$.
  Using this inequality and $1/q=1/p-(\alpha+\beta)/n$,
  we deduce $\|\chi_{Q}\|_X\lesssim \|\chi_{Q}\|_Y\mu(Q)|Q|^{\alpha/n-1}$ for all cubes $Q$.
  Then the desired conclusion immediately follows from Theorem \ref{theorem, structure, homogeneous kernel} and Proposition \ref{proposition, technique, homogeneous kernel, Z=L^infty}.
\end{proof}

Moreover, we can get the following characterization of ${\rm Lip}_{\beta,\omega}$ via boundedness of $[b,T_\alpha]$ in weighted Lebesgue spaces, which is an essential improvement of Theorems 1.1 and 1.2 in \cite{HG-2008}.

\begin{corollary}
  Let $1<p,q <\infty$, $1/q=1/p-(\alpha+\beta)/n$, $0<\beta<1$, $0\le\alpha<n$ with $0<\alpha+\beta<n$, $\omega\in A_1$ and $T_{\alpha}$ be the integral operator associated to the kernel $\Omega$ with nonzero homogeneous function of degree 0 and satisfying (\ref{mean value zero}) for $\alpha=0$. Suppose that $\Omega\in C(\mathbb{S}^{n-1})$. Then for $b\in L_{\rm loc}^1(\mathbb{R}^n)$, we have
\begin{equation*}
b\in {\rm Lip}_{\beta,\omega}\,\Longleftrightarrow\,[b,T_{\alpha}]:\,L^p(\omega)\to L^q(\omega^{1-(1-\alpha/n)q}).
\end{equation*}
\end{corollary}

\begin{proof}
($\Longrightarrow$):\quad For $\Omega\in C(\mathbb{S}^{n-1})$, by the similar arguments in the proofs of \cite[Theorem 1.1 (a) and Theorem 1.2 (a)]{HG-2008}, we can get the desired result. Here we omit the details.

($\Longleftarrow$):\quad Notice that if $\Omega\in C(\mathbb{S}^{n-1})$, then
\begin{equation*}
    \lim_{r\rightarrow 0}\frac{1}{\sigma(B(x',r)\cap \mathbb{S}^{n-1})}\int_{B(x',r)\cap \mathbb{S}^{n-1}}|\Omega(z')-\Omega(x')|d\sigma(z')=0.
  \end{equation*}
Therefore, the desired conclusion directly follows from Corollary \ref{cor.18}.

\end{proof}

\subsection{Commutators on Morrey spaces}
Let $p\in (0,\infty)$. We call that $f\in L_{\rm loc}^p$ belongs to Morrey space $M^{p,\lambda}$ with $-n/p\leq \lambda<0$, if
\begin{equation*}
  \|f\|_{M^{p,\lambda}}:=\sup_{Q}\frac{1}{|Q|^{\lambda/n}}\left(\frac{1}{|Q|}\int_{Q}|f(y)|^pdy\right)^{1/p}<\infty,
\end{equation*}
where $Q$ denotes any cube contained in $\mathbb{R}^n$. When $\lambda=-p/\lambda$, $M^{p,\lambda}(\mathbb{R}^n)$ coincides with the Lebesgue space $L^p(\mathbb{R}^n)$.

It is well known that the Morrey space $M^{p,\lambda}(\mathbb{R}^n)$ is connected to certain problems in elliptic PDEs and was first introduced by Morrey in \cite{Morrey}. Later on, the Morrey spaces were found to have many applications to the Navier-Stokes equations, the Schr\"odinger equations, elliptic equations and potentia analysis etc. (see \cite{CF-1987, CFL-1991, CFL-1993,FLG-1998, Kato-1992, RV-1991,Shen-2003} et al.).

Here, as applications of our main results, we focus on the characterization of $BMO_\mu$ via booundedness of commutators in Morrey spaces. At first, we present an auxiliary lemma as follows.

\begin{lemma}\label{lem4.17}
  Suppose $0<p<\infty$, $-n/p\leq \lambda<0$, $1/\widetilde{p}=1+1/p$, $\widetilde{\lambda}=\lambda-n$.
  Let $Y=M^{p,\lambda}$, $\widetilde{Y}=M^{\widetilde{p},\widetilde{\lambda}}$, $Z=L^1$. Then
\begin{equation*}
  Y\cdot Z\subset\widetilde{Y},\hspace{6mm}\|\chi_{Q}\|_{\widetilde{Y}}\sim \|\chi_{Q}\|_{Y}\cdot \|\chi_{Q}\|_{Z}.
\end{equation*}
\end{lemma}
\begin{proof}
A direct calculation yields that
\begin{equation}
  \begin{split}
    \|fg\|_{M^{\widetilde{p},\widetilde{\lambda}}}
    = &
    \sup_{Q}\frac{1}{|Q|^{\widetilde{\lambda}/n}}\left(\frac{1}{|Q|}\int_{Q}|f(y)g(y)|^{\widetilde{p}}dy\right)^{1/\widetilde{p}}
    \\
    \lesssim &
    \sup_{Q}\frac{1}{|Q|^{\widetilde{\lambda}/n+1/\widetilde{p}}}\left(\int_{Q}|f(y)|^{p}dy\right)^{1/p}
    \cdot \|g\|_{L^1}
    \\
    \sim &
    \sup_{Q}\frac{1}{|Q|^{\lambda/n}}\left(\frac{1}{|Q|}\int_{Q}|f(y)|^{p}dy\right)^{1/p}\cdot \|g\|_{L^1}
    =
    \|f\|_{M^{p,\lambda}}\cdot \|g\|_{L^1}.
  \end{split}
\end{equation}
Thus,
\begin{equation*}
  Y\cdot Z=M^{p,\lambda}\cdot L^1\subset M_{\widetilde{p},\widetilde{\lambda}}=\widetilde{Y}.
\end{equation*}
More over, for any cube $Q$,
\begin{equation*}
  \|\chi_Q\|_{M^{p,\lambda}}\sim |Q|^{-\lambda/n}.
\end{equation*}
Thus,
\begin{equation*}
  \begin{split}
    \|\chi_Q\|_{\widetilde{Y}}\sim |Q|^{-\widetilde{\lambda}/n}
    =
    |Q|\cdot |Q|^{-\lambda/n}
    \sim \|\chi_Q\|_{Y}\|\chi_Q\|_{Z}.
  \end{split}
\end{equation*}
This completes the proof of Lemma \ref{lem4.17}.
\end{proof}

As an application of Proposition \ref{proposition, structure and technique for Z=L^1},
we can establish the following result, which is an essential improvement of \cite[Theorem 1.1]{SL-2013} and \cite[Theorem 1.2]{ZSH-2015},
in which $\Omega\in C^{\infty}(\mathbb{S}^{n-1})$ was assumed.

\begin{corollary}
Let $0<\beta<1$, $0\le \alpha<n$ with $\alpha+\beta<n$, $1<p,q<\infty$, $1/q=1/p-(\alpha+\beta)/n$, $\alpha+\beta+\lambda<0$, $-n/p\leq \lambda<0$.
Set $\nu=\alpha+\beta+\lambda$.
Suppose $\Omega\in L^{\infty}(\mathbb{S}^{n-1})$ is a nonzero homogeneous function of degree 0 and satisfies (\ref{mean value zero}) for $\alpha=0$. If there exist some open subset $E\subset\mathbb{S}^{n-1}$, and constants $c$ and $C$ such that
 $$c\leqslant \Omega(x')\leqslant C,\quad \forall\, x'\in E,$$
  where $0<c<C$ or $c<C<0$, then the following two statements are equivalent:
\begin{enumerate}[(1)]
  \item $b\in {\rm Lip}_{\beta}(\mathbb{R}^n)$,
  \item $[b,T_{\alpha}]$ is bounded from $L^p(\mathbb{R}^n)$ to $L^{q}(\mathbb{R}^n)$,
  \item $[b,T_{\alpha}]$ is bounded from $M^{p,\lambda}(\mathbb{R}^n)$ to $M^{q,\nu}(\mathbb{R}^n)$.
\end{enumerate}

In particular, if $\Omega\in C(\mathbb{S}^{n-1})$, then the above conclusions holds.
\end{corollary}

\begin{proof}
  $(1)\Rightarrow (2)$ follows by Hardy-Littlewood-Sobolev inequality and following dominated estimate of $T_\alpha$:
  \begin{equation*}
    \begin{split}
      |[b,T_\alpha](f)(x)|
      \leqslant &
      \int_{\mathbb{R}^n}\frac{\Omega(x-y)}{|x-y|^n}|b(x)-b(y)||f(y)|dy
      \\
      \leqslant &
      \|b\|_{{\rm Lip}_{\beta}(\mathbb{R}^n)}\|\Omega\|_{L^{\infty}(\mathbb{S}^{n-1})}\int_{\mathbb{R}^n}\frac{1}{|x-y|^{n-\alpha-\beta}}|f(y)|dy
      \\
      \lesssim &
      I_{\alpha+\beta}|f|(x).
    \end{split}
  \end{equation*}
Moreover, this together with the boundedness of $I_\beta$ in Morrey spaces (see \cite{Peetre_JFA_1969}) implies that $(1)\Rightarrow (3)$.

  Next, we verify $(3)\Rightarrow (1)$. Take $X=M^{p,\lambda}(\mathbb{R}^n)$, $Y=M^{q,\nu}(\mathbb{R}^n)$, $Z=L^1(\mathbb{R}^n)$, $\widetilde{Y}=M^{\tilde{q},\widetilde{\nu}}(\mathbb{R}^n)$,
  where $1/\widetilde{q}=1+1/q$, $\widetilde{\nu}=\nu-n$,
  and choose $\mu(Q)=|Q|^{1+\beta/n}$,
  then the conclusion follows immediately by Proposition \ref{proposition, structure and technique for Z=L^1}
  and Lemma \ref{lem4.17}.

  Finally, $(2)\Rightarrow (1)$ can be verified by taking $X=L^p(\mathbb{R}^n)$, $Y=L^{q}(\mathbb{R}^n)$, $Z=L^1(\mathbb{R}^n)$, $\widetilde{Y}=L^{\tilde{q}}(\mathbb{R}^n)$,
  where $1/\widetilde{q}=1+1/q$,
  and taking $\mu(Q)=|Q|^{1+\beta/n}$ in Proposition \ref{proposition, structure and technique for Z=L^1}.
\end{proof}

\begin{remark}
  We remark that the above corollary is also valid for $\beta=0$, i.e., for replacing ${\rm Lip}_{\beta}$ by $BMO$,
  which gives a new characterization of $BMO(\mathbb{R}^n)$ and can be regarded as an essential improvement of the previous results in \cite{Ding-1997} et al. We leave the details for the interested readers.
\end{remark}

Applying Proposition \ref{proposition, structure and technique for Z=L^1} again, we also have the following result, which is an essential improvement of \cite[Theorem 1.5 and Corollary 1.10 (iv)]{CDH-2016}.
\begin{corollary}
Let $0<\beta\leq 1$, $1<p<\infty$, $-n/p\leq \lambda<0$.
Suppose $\Omega\in L^\infty(\mathbb{S}^{n-1})$ is a nonzero homogeneous function of degree 0 and satisfies (\ref{mean value zero}),
and (\ref{l4.6}) if $\beta=1$. If there exist some open subset $E\subset\mathbb{S}^{n-1}$, and constants $c$ and $C$ such that
 $$c\leqslant \Omega(x')\leqslant C,\quad \forall\, x'\in E,$$
  where $0<c<C$ or $c<C<0$, then for the following two statements:
\begin{enumerate}[(1)]
  \item $b\in {\rm Lip}_{\beta}(\mathbb{R}^n)$,
  \item $[b, T_{-\beta}]$ is bounded from $M^{p,\lambda}(\mathbb{R}^n)$ to $M^{p,\lambda}(\mathbb{R}^n)$,
\end{enumerate}
we have $(1)\Longleftrightarrow (2)$ for $\beta=1$, and $(2)\Longrightarrow (1)$ for $\beta\in (0,1)$.

In particular, if $\Omega\in C(\mathbb{S}^{n-1})$, then the above conclusions hold.

\end{corollary}
\begin{proof}
  Note that
  \begin{equation*}
    \frac{\Omega(x-y)}{|x-y|^{n+\beta}}|b(x)-b(y)|
    \lesssim
    \|b\|_{{\rm Lip}_{\beta}(\mathbb{R}^n)}\|\Omega\|_{L^{\infty}(\mathbb{S}^{n-1})}\cdot \frac{1}{|x-y|^n},
  \end{equation*}
  By (\ref{l4.5}) and \cite[Theorem 1.8]{Chen-Ding-Wang_Canda.J.Math_2002}, we get that $(1)\Rightarrow (2)$ for $\beta=1$.

  To verify $(2)\Longrightarrow (1)$,
  we take $X=Y=M^{p,\lambda}(\mathbb{R}^n)$, $Z=L^1(\mathbb{R}^n)$, $\widetilde{Y}=M^{\tilde{p},\widetilde{\lambda}}(\mathbb{R}^n)$,
  where $1/\widetilde{p}=1+1/p$, $\widetilde{\lambda}=\lambda-n$,
  and take $\alpha=-\beta$, $\mu(Q)=|Q|^{1+\beta/n}$ in Proposition \ref{proposition, structure and technique for Z=L^1},
  then $(2)\Rightarrow (1)$ follows immediately from Proposition \ref{proposition, structure and technique for Z=L^1}
  and Lemma \ref{lem4.17}.
\end{proof}

\end{document}